\documentclass[10pt,a4paper,reqno]{article}
\usepackage[textwidth=125mm,textheight=195mm,paperwidth=170mm,paperheight=240mm]{geometry}
\usepackage[utf8]{inputenc}
\usepackage{ulem}
\usepackage{amsmath,amssymb,amsthm,graphicx}
\usepackage{mathrsfs}
\usepackage{graphicx}
\usepackage{eqparbox}
\usepackage{newtxtext,newtxmath}
\usepackage{titlesec}
\usepackage{etoolbox}
\makeatletter
\patchcmd{\ttlh@hang}{\parindent\z@}{\parindent\z@\leavevmode}{}{}
\patchcmd{\ttlh@hang}{\noindent}{}{}{}
\makeatother
\titleformat*{\section}{\normalsize\bfseries}
\titleformat*{\subsection}{\normalsize\bfseries}
\titleformat*{\subsubsection}{\normalsize\bfseries}
\titleformat*{\paragraph}{\normalsize\bfseries}
\usepackage[frak=euler,bb=boondox,cal=boondoxo,calscaled=1.02]{mathalfa}
\usepackage{stmaryrd}

\usepackage[colorlinks]{hyperref}
\hypersetup{
  pdftitle   = {},
  pdfauthor  = {},
  pdfcreator = {},
  linkcolor  = {orange},
  citecolor  = {orange}
}

\usepackage{etoolbox}
\patchcmd{\thmhead}{(#3)}{#3}{}{} 

\usepackage{mathtools}

\usepackage{tikz}
\usetikzlibrary{matrix,arrows,decorations.markings,shapes}
\tikzset{>=stealth,line width=1.5pt}
\newcommand{\tikzto}{\mathrel{\tikz[baseline]\draw[->,line width=.5pt] (0ex,0.65ex) -- (3ex,0.65ex);}}
\newcommand{\tikzmapsto}   {\mathrel{\tikz[baseline] \draw[|->,line width=.5pt] (0pt,0.65ex) -- (3ex,0.65ex);}}

\newcommand{\toshort}{\mkern1mu{\tikz[baseline]\draw[->,line width=.5pt] (0ex,0.65ex) -- (2ex,0.65ex);}\mkern1mu}

\newcommand{\tikzleftarrow}{\mathrel{\tikz[baseline]\draw[<-,line width=.5pt] (0ex,0.65ex) -- (3ex,0.65ex);}}

\newcommand{\tikzrightarrow}{\mathrel{\tikz[baseline]\draw[->,line width=.5pt] (0ex,0.65ex) -- (3ex,0.65ex);}}
\newcommand{\smallto}{\mathrel{\tikz[baseline]  \draw[->,line width=.35pt] (0ex,0.42ex) -- (1.5ex,0.42ex);}}
\newcommand{\toarg}[1]{\mathrel{\tikz[baseline] \path[->,line width=.5pt] (0ex,0.65ex) edge node[above=-.5ex, overlay, font=\scriptsize] {$#1$} (3.5ex,.65ex);}}
\newcommand{\leftarrowarg}[1]{\mathrel{\tikz[baseline] \path[<-,line width=.5pt] (0ex,0.65ex) edge node[above=-.5ex, overlay, font=\scriptsize,pos=.63] {$#1$} (3.5ex,.65ex);}}

\newcommand{\hdot}{{\mkern1.5mu\tikz[baseline] \draw [line width=1pt, line cap=round, fill=black] (0,0.5ex) circle(0.12ex);}}

\numberwithin{equation}{section}
\newtheorem{theorem}[equation]{Theorem}
\newtheorem*{theorem*}{Theorem}
\newtheorem{proposition}[equation]{Proposition}
\newtheorem{lemma}[equation]{Lemma}
\newtheorem*{lemma*}{Lemma}
\newtheorem{corollary}[equation]{Corollary}
\newtheorem*{corollary*}{Corollary}
\theoremstyle{definition}
\newtheorem{definition}[equation]{Definition}

\newtheorem{notation}[equation]{Notation}
\theoremstyle{remark}
\newtheorem{remark}[equation]{Remark}

\DeclareMathOperator{\Hom}{Hom}

\DeclareMathOperator{\Tot}{Tot}
\DeclareMathOperator{\Coh}{Coh}

\DeclareMathOperator{\QCoh}{QCoh}

\DeclareMathOperator{\id}{id}
\DeclareMathOperator{\HH}{HH}
\DeclareMathOperator{\MC}{MC}
\renewcommand{\H}{{\operatorname H}}
\numberwithin{equation}{section}
\newcommand{\blank}{-}
\renewcommand{\dotsc}{... \mkern2.5mu}
\renewcommand{\cdots}{\cdot \mkern-4mu \cdot \mkern-4mu \cdot}
\renewcommand{\phi}{\varphi}
\newcommand{\Oplus}{\mathbin{\mkern3.5mu\oplus\mkern3.5mu}}

\newcommand{\extprod}{\Lambda\mkern-1.5mu}

\newcommand{\lmb}{\langle}
\newcommand{\rmb}{\rangle}

\newcommand{\hair}{\ifmmode\mskip1.5mu\else\kern0.05em\fi}

\newcommand{\llrr}[1]{%
\llbracket #1 \rrbracket}

\renewcommand{\epsilon}{\varepsilon}
\renewcommand{\theta}{\vartheta}

\newcommand{\hatotimes}{\mathbin{\widehat{\otimes}}}

\usepackage{enumitem}
\setenumerate[0]{label=(\textit{\roman*}\hair),font=\rm}

\begin{document}
\title{Deformation--obstruction theory for diagrams of algebras and applications to geometry}

\author{Severin Barmeier and Yaël Frégier}

\date{}

\maketitle

\begin{abstract}
Let $X$ be an algebraic variety over an algebraically closed field of characteristic $0$ and let $\Coh (X)$ denote its Abelian category of coherent sheaves. By the work of W.~Lowen and M.~Van den Bergh, it is known that the deformation theory of $\Coh (X)$ as an Abelian category can be seen to be controlled by the Gerstenhaber--Schack complex associated to the restriction of the structure sheaf $\mathcal O_X \vert_{\mathfrak U}$ to a cover of affine open sets. We construct an explicit L$_\infty$ algebra structure on the Gerstenhaber--Schack complex controlling the higher deformation theory of $\mathcal O_X \vert_{\mathfrak U}$ or $\Coh (X)$ in case $X$ can be covered by two acyclic open sets, giving an explicit defor\-mation--ob\-struc\-tion calculus for such deformations. For smooth $X$, such deformations recover the deformation of complex structures and deformation quantizations of $X$ as degenerate cases, as we show by means of concrete examples.

\end{abstract}

\tableofcontents

\section{Introduction}

Noncommutative instantons were first studied over a noncommutative $\mathbb R^4$ by Nekrasov and Schwarz \cite{nekrasovschwarz} and have since attracted a lot of attention in the physical literature. In \cite{barmeiergasparim1,barmeiergasparim2} we study (noncommutative) instantons on four-manifolds with nontrivial topology via complex geometry, by identifying instantons with (framed stable) holomorphic rank~$2$ bundles via a Kobayashi--Hitchin correspondence for the noncompact complex surfaces $Z_k = \Tot \mathcal O_{\mathbb P^1} (-k)$. In particular, viewing an instanton as a locally free (thus coherent) sheaf of $\mathcal O_X$ modules one obtains noncommutative instantons by considering noncommutative deformations of $\mathcal O_X$ as a presheaf, which can also be viewed as deformations of its category $\Coh (X)$ of coherent sheaves {\it as an Abelian category} \cite{dinhvanliulowen,lowenvandenbergh1,lowenvandenbergh2}. Our aim in this paper is to develop general tools to control this deformation theory.

In \cite{gerstenhaberschack} M.~Gerstenhaber and S.\hair D.~Schack developed a deformation theory for {\it diagrams of (associative) algebras}, controlled by the {\it Gerstenhaber--Schack complex} $\mathrm C^\hdot_{\mathrm{GS}}$ (see Definition \ref{definitiongerstenhaberschack}). Here, a {\it diagram of associative $\mathbb k$-algebras} over a small category $\mathfrak U$ is a functor $\mathbb A \colon \mathfrak U^{\mathrm{op}} \tikzto \mathfrak{Alg}_{\mathbb k}$.

Denoting by $\H_{\mathrm{GS}}^\hdot$ the cohomology groups of $\mathrm C^\hdot_{\mathrm{GS}}$, deformations of a diagram $\mathbb A$ of associative algebras are parametrized by $\H^2_{\mathrm{GS}} (\mathbb A)$ with obstructions lying in $\H^3_{\mathrm{GS}} (\mathbb A)$. This deformation theory generalizes the usual deformation theory of algebras due to Gerstenhaber, as a ``single'' algebra $A$ can be thought of as a diagram over the trivial category. Indeed, for a single algebra $A$, the Gerstenhaber--Schack complex $\mathrm C^\hdot_{\mathrm{GS}} (A)$ coincides with the Hochschild complex $\mathrm{CH}^\hdot (A)$, which is known to control the deformation theory of $A$ as an associative algebra.

It is a remarkable fact \cite{gerstenhaberschack} that to a diagram of algebras $\mathbb A$, one can associate a single associative algebra $\mathbb A!$, its {\it diagram algebra}, such that there is an isomorphism of the Hochschild cohomology $\HH^n (\mathbb A!)$ with the Gerstenhaber--Schack cohomology $\H^n_{\mathrm{GS}} (\mathbb A)$. 

Now let $X$ be a Noetherian semi-separated scheme over an algebraically closed field $\mathbb k$ of characteristic $0$, and let $\Coh (X)$ denote its Abelian category of coherent sheaves. Let $\mathfrak U = \{ U_i \}_{i \in I}$ be a semi-separating cover of $X$, {\it i.e.}\ a cover of acyclic open sets which is closed under intersections. We may think of $\mathfrak U$ as a subcategory of $\mathfrak{Open} (X)$, with objects $U_i$ and morphisms given by inclusion of open sets.

Any presheaf of algebras $\mathcal F$ on $X$ gives rise to a diagram of algebras $\mathcal F \vert_{\mathfrak U}$ over $\mathfrak U$ obtained by restriction of $\mathcal F \colon \mathfrak{Open} (X)^{\mathrm{op}} \tikzto \mathfrak{Alg}_{\mathbb k}$ to the subcategory $\mathfrak U \subset \mathfrak{Open} (X)$.

In \cite{lowenvandenbergh1,lowenvandenbergh2} W.\ Lowen and M.\ Van den Bergh developed a deformation theory for (abstract) Abelian categories and showed that for the Abelian category $\Coh (X)$ of coherent sheaves on a Noetherian semi-separated scheme $X$, the complex controlling Abelian deformations of $\Coh (X)$ is isomorphic to the Gerstenhaber--Schack complex for the diagram of algebras $\mathcal O_X \vert_{\mathfrak U}$ in the homotopy category of B$_\infty$ algebras. In particular, the deformation theory of $\Coh (X)$ as an Abelian category can be described by higher structures on the Gerstenhaber--Schack complex, which also controls deformations of the diagram of algebras $\mathcal O_X \vert_{\mathfrak U}$, the restriction of the structure sheaf $\mathcal O_X$ to a semi-separating cover $\mathfrak U$.

Moreover, Dinh Van--Lowen \cite{dinhvanlowen} constructed a homotopy which can be used to transfer the \textsc{dg} Lie algebra structure on the Hochschild complex $\mathrm{CH}^\hdot (\mathbb A!)$ to an L$_\infty$ algebra structure on the Gerstenhaber--Schack complex $\mathrm C_{\mathrm{GS}}^\hdot (\mathbb A)$.

In this paper we give an explicit construction of an L$_\infty$ algebra structure on the Ger\-sten\-haber--Schack complex via a different method and prove that the higher structures describe the higher deformation theory of diagrams of associative algebras over
$\mathfrak U =
\begin{tikzpicture}[baseline]
\draw[line width=.5pt, fill=black]   (0,.65ex) circle(0.25ex);
\draw[line width=.5pt, fill=black] (4ex,.65ex) circle(0.25ex);
\draw[line width=.5pt, fill=black] (8ex,.65ex) circle(0.25ex);
\node[shape=circle,scale=.6](U) at   (0,.65ex) {};
\node[shape=circle,scale=.6](W) at (4ex,.65ex) {};
\node[shape=circle,scale=.6](V) at (8ex,.65ex) {};
\draw[->,line width=.5pt] (W) -- (U);
\draw[->,line width=.5pt] (W) -- (V);
\end{tikzpicture}$. Our approach is based on {\it higher derived brackets} due to Voronov \cite{voronov1,voronov2}, which were also used in Frégier--Zambon \cite{fregierzambon1,fregierzambon2}, where the authors study several ``simultaneous deformation'' problems in algebra and geometry, such as simultaneous deformations of two Lie algebras and a morphism between them, of coisotropic submanifolds in Poisson manifolds, or of Dirac structures in Courant algebroids; these methods were also studied from an operadic point of view in Frégier--Markl--Yau \cite{fregiermarklyau}.

The L$_\infty$ algebra structure on $\mathrm C_{\mathrm{GS}}^\hdot$ allows us to give an explicit description of the defor\-mation--ob\-struc\-tion calculus of the diagram $\mathcal O_X \vert_{\mathfrak U}$ in case $X$ can be covered by two acyclic open sets.

In general, deformations of $\Coh (X)$ or of $\mathcal O_X \vert_{\mathfrak U}$ can be thought of as organizing deformations of the complex structure of $X$ and deformation quantizations of the structure sheaf $\mathcal O_X$ into one consistent whole.

The Gerstenhaber--Schack complex of the diagram $\mathcal O_X \vert_{\mathfrak U}$ computes the Hochschild cohomology $\HH^n (X)$ of $X$ \cite{gerstenhaberschack}. In case $X$ is smooth, the Hochschild--Kostant--Rosenberg theorem (see \cite[\S 28]{gerstenhaberschack} and \cite[\S 3]{dinhvanliulowen}) decomposes into a direct sum
\[
\HH^n (X) \simeq \bigoplus_{p+q = n} \H^q (X, \extprod^p \mathcal T_X)
\]
where $\mathcal T_X$ is the tangent bundle of $X$. Deformations of $\Coh (X)$ are parametrized by \cite{lowenvandenbergh2,dinhvanliulowen}
\begin{equation}
\label{hkrdecomposition2}
\H^2_{\mathrm{GS}} (X) \simeq \HH^2 (X) \simeq \H^0 (X, \extprod^2 \mathcal T_X) \oplus \H^1 (X, \mathcal T_X) \oplus \H^2 (X, \mathcal O_X)
\end{equation}
with obstructions lying in
\begin{equation}
\label{hkrdecomposition3}
\H^3_{\mathrm{GS}} (X) \simeq \HH^3 (X) \simeq \H^0 (X, \extprod^3 \mathcal T_X) \oplus \H^1 (X, \extprod^2 \mathcal T_X) \oplus \H^2 (X, \mathcal T_X) \oplus \H^3 (X, \mathcal O_X).
\end{equation}

The summands of (\ref{hkrdecomposition2}) have the following geometric interpretation:
\begin{enumerate}
\item $\H^0 (X, \extprod^2 \mathcal T_X)$ is the space of almost Poisson structures. An almost Poisson structure $\eta$ is a Poisson structure in case its Schouten--Nijenhuis bracket $[\eta, \eta] \in \H^0 (X, \extprod^3 \mathcal T_X)$ vanishes. Such a Poisson structure $\eta$ is seen to parametrize noncommutative deformations of $\mathcal O_X$ in the sense of deformation quantization (Kontsevich \cite{kontsevich1,kontsevich2}) \label{item1}
\item $\H^1 (X, \mathcal T_X)$ parametrizes ``classical'' deformations of the complex structure (after Kodaira--Spencer \cite{kodaira}) \label{item2}
\item $\H^2 (X, \mathcal O_X)$ parametrizes ``twists'' of $\mathcal O_X$ (see \cite{calaquerossi,dinhvanliulowen}).
\end{enumerate}
The summands $\H^0 (X, \extprod^3 \mathcal T_X)$ and $\H^2 (X, \mathcal T_X)$ are seen to be the obstruction spaces for the types \ref{item1} and \ref{item2}, respectively.

The L$_\infty$ algebra structure on the Gerstenhaber--Schack complex $\mathrm C^\hdot_{\mathrm{GS}}$ now controls how these types of deformations interact when considering deformations to higher orders. For example, a variety may admit unobstructed deformations both in the ``classical'' and in the noncommutative sense, which do not extend to a simultaneous deformation. We give such examples in \S \ref{deformationsofcoh} with explicit computations for the noncompact surfaces $Z_k = \Tot \mathcal O_{\mathbb P^1} (-k)$, for $k \geq 1$.

\section{The Gerstenhaber--Schack complex}

Let $\mathbb k$ be a field of characteristic $0$ and write $\Hom = \Hom_{\mathbb k}$ and $\otimes = \otimes_{\mathbb k}$.

\begin{definition}
\label{definitiongerstenhaberschack}
Given a presheaf $\mathcal F$ over a small category $\mathfrak U$, consider the following first quadrant double complex
\[
\mathrm C^{p,q} (\mathcal F) = \prod_{U_0 \smallto \dotsb \smallto U_p} \Hom (\mathcal F (U_p)^{\otimes q}, \mathcal F (U_0))
\]
where the product is taken over all $p$-simplices in the simplicial nerve of $\mathfrak U$. The differentials of $\mathrm C^{p,q}$ are the Hochschild and simplicial differentials
\begin{align*}
d_{\mathrm H} &\colon \mathrm C^{p,q} (\mathcal F) \tikzto \mathrm C^{p,q+1} (\mathcal F) \\
d_\Delta &\colon \mathrm C^{p,q} (\mathcal F) \tikzto \mathrm C^{p+1,q} (\mathcal F).
\end{align*}
(These differentials are described in detail in Appendix \ref{hochschildsimplicial}.)
The {\it Gerstenhaber--Schack complex} is defined as the total complex $\mathrm C^\hdot_{\mathrm{GS}} (\mathcal F) = \Tot \mathrm C^{p,q} (\mathcal F)$ and the {\it Gerstenhaber--Schack differential} is the usual total differential $d_{\mathrm{GS}} = d_{\mathrm H} + (-1)^{q+1} d_\Delta$.
\end{definition}

\begin{remark}
As a category, $\mathfrak U$ contains identity morphisms and its simplicial nerve $\mathrm N (\mathfrak U)$ contains simplices of length ${>}1$. There is a subcomplex, the {\it reduced} Ger\-sten\-haber--Schack complex, consisting of those morphisms which vanish on simplices containing an identity arrow and the inclusion in the Gerstenhaber--Schack complex is a quasi-isomorphism, see \cite[\S 3.4]{dinhvanlowen}. Henceforth we shall work with the reduced Gerstenhaber--Schack complex.
\end{remark}

\begin{remark}
Gerstenhaber--Schack \cite{gerstenhaberschack} found that the total complex of the truncated complex with $q \geq 1$ parametrizes deformations of $\mathcal F$ as a presheaf of algebras.
Dinh Van--Lowen \cite{dinhvanlowen} showed that the full Gerstenhaber--Schack complex parametrizes deformations of a sheaf as a {\it twisted presheaf}, and can also be generalized to describe deformations of {\it prestacks}. (We note that to this end Dinh Van--Lowen defined a more complicated differential on $\mathrm C^\hdot_{\mathrm{GS}}$.\footnote{Their new differential is of the form $d^{p+q} = d_0 + \dotsb + d_{p+q}$, where $d_0 = d_{\mathrm H}$ and $d_1 = d_\Delta$ and the other terms are maps $d_i \colon \mathrm C^{p,q} \tikzto \mathrm C^{p+i,q-i+1}$.} However, as long as $X$ is covered by two acyclic open sets with acyclic intersection, this extra structure does not appear.)

For our applications it thus suffices to only consider the truncated (reduced) Gersten\-haber--Schack complex, which we will also denote by $\mathrm C^\hdot_{\mathrm{GS}}$.
\end{remark}

\subsection{Deformations of algebras and their diagrams}

We briefly recall the deformation theory of a ``single'' associative $\mathbb k$-algebra $A = (A, \mu)$ first studied by Gerstenhaber \cite{gerstenhaber1,gerstenhaber2}. Its multiplication $\mu \colon A \otimes A \tikzto A$ can be viewed as an element of degree two\footnote{To view $\mathrm{CH}^\hdot (A)$ as a \textsc{dg} Lie algebra, one shifts the degree by $1$, {\it cf.}\ Proposition \ref{associative}.} in the Hochschild complex $\mathrm{CH}^\hdot (A)$. Let $A \llrr{\epsilon} = A \hatotimes \mathbb k \llrr{\epsilon}$ and consider the natural extension of $\mu$ to $A \llrr{\epsilon}$, given by
\[
\mu \Big( \textstyle\sum_{i \geq 0} a_i \epsilon^i, \textstyle\sum_{j \geq 0} b_j \epsilon^j \Big) = \sum_{k \geq 0} \Big( \textstyle\sum_{i + j = k} \mu (a_i, b_j) \Big) \epsilon^k.
\]

A {\it formal deformation} of $A$ is a collection $(\mu_i)_{i \geq 1}$ of $\mathbb k$-bilinear maps $\mu_i \colon A \otimes A \tikzto A$ such that their extensions to $A \llrr{\epsilon}$ define a $\mathbb k \llrr{\epsilon}$-bilinear associative multiplication $\mu^\epsilon \colon A \llrr{\epsilon} \times A \llrr{\epsilon} \tikzto A \llrr{\epsilon}$ of the form
\[
\mu^\epsilon = \mu + \epsilon \mu_1 + \epsilon^2 \mu_2 + \dotsb
\]

The deformation theory of $A$ can be conveniently described in terms of a graded Lie algebra structure on the Hochschild cochains defined as follows.

\begin{definition}
\label{gerstenhaberbracket}
Given two multilinear maps $f \in \Hom (A^{\otimes m+1}, A)$, $g \in \Hom (A^{\otimes n+1}, A)$
define
\begin{align*}
f \circ_i g &= f (\id^{\otimes i} \otimes \, g \otimes \id^{\otimes m-i})
\intertext{for $0 \leq i \leq m$ and write}
f \circ g &= \sum_{i=0}^m (-1)^{ni} f \circ_i g.
\intertext{The {\it Gerstenhaber bracket} is then defined by}
[f, g] &= f \circ g - (-1)^{mn} g \circ f.
\end{align*}
\end{definition}

\begin{proposition}[\cite{gerstenhaber1}]
\label{associative}
Let $\mu \in \Hom (A \otimes A, A)$. Then
\[
\mu \text{ \rm is associative} \Leftrightarrow [\mu, \mu] = 0.
\]
\end{proposition}

\begin{corollary}
$\mu^\epsilon$ is an associative deformation of $\mu$ if and only if $\widetilde \mu = \mu^\epsilon - \mu$ satisfies the Maurer--Cartan equation
\begin{equation}
\label{maurercartanassociative}
d_{\mathrm H} \widetilde \mu + \tfrac12 [\widetilde \mu, \widetilde \mu] = 0.
\end{equation}
\end{corollary}

Here $d_{\mathrm H} = [\mu, \blank ]$ together with the Gerstenhaber bracket $[\blank {,} \blank]$ define a \textsc{dg} Lie algebra structure on the Hochschild cochains. In this sense deformations of an associative algebra are governed by a \textsc{dg} Lie algebra. In particular, one can construct a deformation of $A$ term by term using (\ref{maurercartanassociative}).

In the following, we wish to find a similar structure for deformations of more general diagrams.

Now consider a diagram $\mathbb A$ over $\mathfrak U = \big( U \tikzleftarrow W \tikzrightarrow V \big)$ and let
\begin{itemize}
\item $M_0 = (\mu_0, \nu_0, \xi_0) \in \mathrm C^{0,2}$ be the multiplications on $\mathbb A (U), \mathbb A (V), \mathbb A (W)$, respectively,
\item $\Phi_0 = (\phi_0, \psi_0) \in \mathrm C^{1,1}$ be the morphisms $\mathbb A (W {\toshort} U)$, $\mathbb A (W {\toshort} V)$, respectively.
\end{itemize}
(Note that there are no ``twists'' since $\mathrm C^{2,0}$ is zero in the reduced Gerstenhaber--Schack complex.)

A {\it formal deformation} of $\mathbb A$ is a diagram $\mathbb A \llrr{\epsilon}$ such that for each object $U \in \mathfrak U$, $\mathbb A \llrr{\epsilon} (U) = \mathbb A (U) \hatotimes \mathbb k \llrr{\epsilon}$ is a deformation of $\mathbb A (U)$ with multiplication $\mu^\epsilon = \mu + \epsilon \mu_1 + \epsilon^2 \mu_2 + \dotsb$ and for each morphism $W \tikzto U$ in $\mathfrak U$, $\mathbb A \llrr{\epsilon} (W \toshort U) = \phi + \epsilon \phi_1 + \epsilon^2 \phi_2 + \dotsb$ is a deformation of the morphism $\phi = \mathbb A (W \toshort U)$.

\section{L$_\infty$ algebras via higher derived brackets}

An {\it L$_\infty$ algebra}\footnote{L$_\infty$ algebras are also called {\it strongly homotopy} (or \textsc{sh}) Lie algebras} is a graded vector space together with a collection of $n$-ary ``brackets'' satisfying graded anti-symmetry and generalized Jacobi identities.

We find it convenient to shift the grading and work with what one may call an L$_\infty [1]$ algebra. The suspension of an L$_\infty [1]$ algebra is again an (ordinary) L$_\infty$ algebra.

\begin{definition}
\label{linfinityalgebra}
An {\it L$_\infty [1]$ algebra} $(\mathfrak g, \{ m_n \}_{n \geq 0})$ is a graded $\mathbb k$-vector space $\mathfrak g = \prod_{m \in \mathbb Z} \mathfrak g^m$ together with a collection of multilinear maps $m_n \colon \mathfrak g^{\otimes n} \tikzto \mathfrak g$ of degree $1$ satisfying
\begin{enumerate}
\item $m_n (x_{s (1)}, \dotsc, x_{s (n)}) = \epsilon (s) \, m_n (x_1, \dotsc, x_n)$ for any $s \in \mathfrak S_n$
{\vspace{.2em}\leavevmode\unskip\nobreak\hfil\penalty50\hskip-1em
  \hbox{}\nobreak\hfil{\it (graded anti-symmetry)\vspace{.5ex}}%
  \parfillskip=0pt \finalhyphendemerits=0 \endgraf}
\item $\displaystyle\sum_{\substack{i+j=n+1 \\ i,j\geq 1}} \sum_{s \in \mathfrak S_{i,n-i}} \epsilon (s) \, m_j (m_i (x_{s (1)}, \dotsc, x_{s (i)}), x_{s (i+1)}, \dotsc, x_{s (n)}) = 0$
{\vspace{-1.75em}\leavevmode\unskip\nobreak\hfil\penalty50\hskip-1em
  \hbox{}\nobreak\hfil{\it (generalized Jacobi identity)\vspace{1ex}}%
  \parfillskip=0pt \finalhyphendemerits=0 \endgraf}
\end{enumerate}
for homogeneous elements $x_1, \dotsc, x_n$. Here
\begin{itemize}
\item $\mathfrak S_n$ is the set of permutations of $n$ elements
\item $\mathfrak S_{i,n-i} \subset \mathfrak S_n$ is the set of {\it unshuffles}, {\it i.e.}\ permutations $s \in \mathfrak S_n$ satisfying $s (1) < \dotsb < s (i)$ and $s (i+1) < \dotsb < s (n)$
\item $\epsilon (s)$ is the {\it Koszul sign}\footnote{The Koszul sign of a transposition of two elements $x_i, x_j$ is defined by $(-1)^{\lvert x_i \rvert \lvert x_j \rvert}$, where $\lvert x_i \rvert$ denotes the degree of $x_i$. This definition is then extended multiplicatively to an arbitrary permutation using a decomposition into transpositions.} of the permutation $s$, which also depends on the degrees of the $x_i$.
\end{itemize}
We denote the $n$-ary multilinear maps $m_n (\blank, \dotsc {,} \blank )$ by $\lmb \blank , \dotsc {,} \blank \rmb$.
\end{definition}

The first few generalized Jacobi identities, starting at $n = 0$, read
\begin{align*}
\lmb \lmb\rmb \rmb &= 0 \\
\lmb \lmb a \rmb \rmb + \lmb \lmb\rmb, a \rmb &= 0 \\
\lmb \lmb a, b \rmb \rmb + \lmb \lmb a \rmb, b \rmb + (-1)^{\lvert a \rvert \lvert b \rvert} \lmb \lmb b \rmb, a \rmb + \lmb \lmb\rmb, a, b \rmb &= 0
\end{align*}
For $\lmb\rmb = 0$\footnote{The $0$-ary ``bracket'' $\lmb\rmb$ of an L$_\infty$ algebra is simply a distinguished element. If this element is non-zero, the L$_\infty$ is said to be {\it curved}.
However, in what follows we will not need curved L$_\infty$ algebras.} and writing $\lmb \blank \rmb = d$, these identities start at $n = 1$ and read
\begin{align}
(d \circ d) (a) &= 0 \label{jacobiidentities1} \\
d \lmb a, b \rmb + \lmb d a, b \rmb + (-1)^{\lvert a \rvert \lvert b \rvert} \lmb d b, a \rmb &= 0 \label{jacobiidentities2}\\
\lmb \lmb a, b \rmb, c \rmb + (-1)^{\lvert b \rvert \lvert c \rvert} \lmb \lmb a, c \rmb, b \rmb + (-1)^{\lvert a \rvert \lvert b \rvert + \lvert a \rvert \lvert c \rvert} \lmb \lmb b, c \rmb, a \rmb \hspace{6em} \notag \\ {} + d \lmb a, b, c \rmb + \lmb d a, b, c \rmb + (-1)^{\lvert a \rvert \lvert b \rvert} \lmb d b, a, c \rmb + (-1)^{\lvert a \rvert \lvert c \rvert + \lvert b \rvert \lvert c \rvert} \lmb d c, a, b \rmb &= 0 \label{jacobiidentities3}
\end{align}
{\it i.e.}\ $d$ is a differential (\ref{jacobiidentities1}), $d$ is a derivation with respect to the binary bracket (\ref{jacobiidentities2}) and the usual (shifted) Jacobi identity holds for the binary bracket (first line of (\ref{jacobiidentities3})) {\it up to homotopy correction terms} (second line of (\ref{jacobiidentities3})).

\begin{remark}
\label{vanishinghigherbrackets}
If the $n$-ary brackets are identically zero for $n > 1$, one obtains a cochain complex with differential $d = \lmb \blank \rmb$; if the brackets are zero for $n > 2$, one obtains a \textsc{dg} Lie algebra.
\end{remark}

\begin{definition}
Given an L$_\infty [1]$ algebra $(\mathfrak g, \lmb\rmb, \lmb \blank \rmb, \lmb \blank{,} \blank \rmb, \dotsc)$, a {\it Maurer--Cartan element} is an element $\Phi$ of degree $0$ satisfying the Maurer--Cartan equation
\[
\exp_{\lmb \, \rmb} \Phi = \sum_{n=0}^\infty \frac{\Phi^{\lmb n \rmb}}{n!} = 0
\]
where $\Phi^{\lmb n \rmb} = \lmb \Phi, \dotsc, \Phi \rmb$ is the bracket of $n$ copies of $\Phi$.
Denote by $\MC (\mathfrak g) \subset \mathfrak g^0$ the set of Maurer--Cartan elements of $\mathfrak g$.
\end{definition}

\begin{remark}
When the $n$-ary brackets are zero for $n > 2$ and the L$_\infty$ algebra is in fact a \textsc{dg} Lie algebra with differential $d = \lmb \blank \rmb$ and \textsc{dg} Lie bracket $\lmb \blank {,} \blank \rmb$, the Maurer--Cartan equation for an element $\alpha$ reduces to the familiar form
\[
d \alpha + \tfrac12 \lmb \alpha, \alpha \rmb = 0.
\]
\end{remark}

\subsection{Voronov's higher derived brackets}

From the definition we gave, constructing non-trivial examples of L$_\infty$ algebras might seem like a daunting task, as it involves infinitely many multilinear maps satisfying infinitely many compatibility conditions. Here we follow \cite{fregierzambon1,fregierzambon2} and use a construction due to Voronov \cite{voronov1,voronov2} (see also \cite{kosmannschwarzbach1,kosmannschwarzbach2}), which constructs an L$_\infty$ algebra from simple data.

\begin{definition}
Let $(\mathfrak g, [\blank {,} \blank])$ be a graded Lie algebra and let $\mathfrak a$ be an Abelian subalgebra. Let $P \colon \mathfrak g \tikzto \mathfrak a$ be a projection such that $\ker P \subset \mathfrak g$ is a subalgebra and let $M \in \ker P \cap \mathfrak g^1$ satisfying $[M, M] = 0$. The data $(\mathfrak g, \mathfrak a, P, M)$ which we shall write visually as $M \in \mathfrak g \toarg{P} \mathfrak a$ are called {\it Voronov data}. If instead $M \in \mathfrak g^1 \setminus \ker P$, then $M \in \mathfrak g \toarg{P} \mathfrak a$ are called {\it curved Voronov data}.
\end{definition}

\begin{theorem}[\cite{voronov1}]
\label{voronov}
Let $M \in \mathfrak g \toarg{P} \mathfrak a$ be (curved) Voronov data. Then
\begin{enumerate}
\item $\mathfrak a_M^P = (\mathfrak a, \lmb \rmb, \lmb \blank \rmb, \lmb \blank {,} \blank \rmb, \dotsc)$ is a (curved) L$_\infty [1]$ algebra with multibrackets defined by
\begin{align*}
\lmb \rmb &= P M \\
\lmb a_1, \dotsc, a_n \rmb &= P [...[[M, a_1], a_2], \dotsc, a_n]
\end{align*}
\item $(\mathfrak g[1] \oplus \mathfrak a)_M^P = (\mathfrak g[1] \oplus \mathfrak a, \lmb \rmb, \lmb \blank \rmb, \lmb \blank {,} \blank \rmb, \dotsc)$ is a (curved) L$_\infty [1]$ algebra with multibrackets defined by
\begin{alignat*}{3}
d (x[1] \oplus a) = \lmb x[1] \oplus a \rmb &= -[M, x][1] \oplus P (x + [M, a]) && \in \mathfrak g[1] \oplus \mathfrak a \\
\lmb x[1], y[1] \rmb &= (-1)^{\lvert x \rvert + 1} [x, y][1] && \in \mathfrak g[1] \\
\lmb x[1], a_1, \dotsc, a_n \rmb &= P [...[[\eqmakebox[xM0]{$x$}, a_1], a_2], \dotsc, a_n]  && \in \mathfrak a \\
\lmb a_1, \dotsc, a_n \rmb &= P [...[[\eqmakebox[xM0]{$M$}, a_1], a_2], \dotsc, a_n] && \in \mathfrak a
\end{alignat*}
where $x[1], y[1] \in \mathfrak g[1]$ and $a_1, \dotsc, a_n \in \mathfrak a$. Up to permutation of the entries all other multibrackets are set to vanish.
\end{enumerate}
\end{theorem}

\begin{remark}
Theorem \ref{voronov} remains true if the inner derivation $[M{,} \blank]$ is replaced by an arbitrary derivation, which was shown in \cite{voronov2}.
\end{remark}

The construction of an L$_\infty$ algebra via derived brackets might appear like a specialized class of examples. However, they are in fact very general: any L$_\infty$ algebra can be given via derived brackets \cite[Prop.\ 2.12]{fregierzambon2}.

\begin{notation}
\label{pphi}
Let $P_\Phi = P \circ \exp [\blank, \Phi] \colon \mathfrak g \tikzto \mathfrak a$.
\end{notation}

\begin{remark}
Let $M \in \mathfrak g \toarg{P} \mathfrak a$ be Voronov data. Then $\Phi \in \mathfrak a^0$ is a Maurer--Cartan element of $\mathfrak a^P_M$ if and only if $M \in \ker P_\Phi$.
\end{remark}

\section{An explicit L$_\infty$ algebra structure on the Gerstenhaber--Schack complex}
\label{linfinitygerstenhaberschack}

Now let $\mathbb A$ be a diagram of algebras over $\mathfrak U = \big( U \tikzleftarrow W \tikzrightarrow V \big)$ and let
\begin{flalign*}
&& \mathbb A \eqmakebox[UVW1]{$(U)$} &= (\eqmakebox[UVW2]{$A_U$}, \mu) & \mu &\in \Hom (A_U^{\otimes 2}, A_U) & \\
&& \mathbb A \eqmakebox[UVW1]{$(V)$} &= (\eqmakebox[UVW2]{$A_V$}, \nu) & \nu &\in \Hom (A_V^{\otimes 2}, A_V) & \\
&& \mathbb A \eqmakebox[UVW1]{$(W)$} &= (\eqmakebox[UVW2]{$A_W$}, \xi) & \xi &\in \Hom (A_W^{\otimes 2}, A_W) &
\end{flalign*}
where $\mu, \nu, \xi$ are associative multiplications.

\begin{remark}
The rest of the section also works for diagrams of algebras over
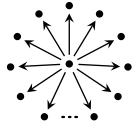
\[
\begin{tikzpicture}[x=2.2em,y=2.2em]
\draw[line width=.5pt, fill=black] (0,0) circle(0.25ex);
\node[shape=circle,scale=.7](C) at (0,0) {};
\foreach \d in {90,121,152,183,214,245,59,28,357,326,295}{%
\draw[line width=.5pt, fill=black] (\d:1) circle(0.25ex);
\node[shape=circle,scale=.6](\d) at (\d:1) {};
\draw[->,line width=.5pt] (C) -- (\d);
\node[font=\scriptsize] at (-90:.95) {$\cdots$};
}
\end{tikzpicture}
\]
However, for our applications we only need to consider the case of two arrows, so we only give details for this case.
\end{remark}

Let $\mathfrak g = \prod_{n \geq 0} \mathfrak g^n$ and $\mathfrak a = \prod_{n \geq 0} \mathfrak a^n$ be defined by
\begin{align}
\mathfrak g^n &= \displaystyle\bigoplus_{U_i \in \mathfrak U} \Hom (\textstyle\bigotimes_{i=0}^n A_{U_i}, A_W) \oplus \Hom (A_U^{\otimes n+1}, A_U) \oplus \Hom (A_V^{\otimes n+1}, A_V) \label{ga}\\
\mathfrak a^n &= \Hom (A_U^{\otimes n+1}, A_W) \oplus \Hom (A_V^{\otimes n+1}, A_W) \notag
\end{align}
and equip $\mathfrak g$ with the Gerstenhaber bracket denoted by $[\blank {,} \blank]$ and defined in Definition \ref{gerstenhaberbracket}. Then $(\mathfrak g, [\blank {,} \blank])$ is a graded Lie algebra and $\mathfrak a \subset \mathfrak g$ an Abelian subalgebra.

Now let $P \colon \mathfrak g \tikzto \mathfrak a$ be the projection given by the decomposition (\ref{ga}) and set $M = \mu \oplus \nu \oplus \xi \in \ker P \cap \mathfrak g^1$. 
Each of $\mu, \nu, \xi$ is only composable with itself and so
\begin{align*}
[M, M] = [\mu, \mu] \oplus [\nu, \nu] \oplus [\xi, \xi] = 0
\end{align*}
since $\mu, \nu, \xi$ are associative ({\it cf.}\ Proposition \ref{associative}).

\begin{lemma}
$M \in \mathfrak g \toarg{P} \mathfrak a$ define Voronov data.
\end{lemma}

\begin{proof}
It remains to check that $\ker P$ is a graded Lie subalgebra of $\mathfrak g$. We can decompose $\ker P \cap \mathfrak g^n$ as
\[
(\ker P)^n = K_U^n \oplus K_V^n \oplus K_W^n
\]
where
\begin{alignat*}{5}
& K^n_U & {}={} & \Hom (A_U^{\otimes n+1}, A_U) \\
& K^n_V & {}={} & \Hom (A_V^{\otimes n+1}, A_V) \\
& K^n_W & {}={} & \bigoplus_{\substack{U_i \in \mathfrak U \\ W \in \{ U_i \}}} \Hom (A_{U_0} \otimes \dotsb \otimes A_{U_n}, A_W).
\end{alignat*}
One easily verifies that $\ker P$ is closed under the compositions $\circ_i$, thus also under the bracket.
\end{proof}

\begin{proposition}
\label{equivalent}
Let $\phi \in \Hom (A_U, A_W)$ and $\psi \in \Hom (A_V, A_W)$ and set $\Phi = -\phi \oplus -\psi \in \Hom (A_U, A_W) \oplus \Hom (A_V, A_W) = \mathfrak a^0$. The following are equivalent
\begin{enumerate}
\item \label{tfae1} $\Phi \in \MC (\mathfrak a_M^P)$
\item \label{tfae2} $P_{\Phi} (M) = 0$
\item \label{tfae3} $\phi$ and $\psi$ are morphisms compatible with the algebra structures on $A_U, A_V, A_W$.
\end{enumerate}
\end{proposition}

\begin{proof}
Recall from Notation \ref{pphi} that $P_\Phi = P \circ \exp [\blank, \Phi]$. Then \ref{tfae1} $\Leftrightarrow$ \ref{tfae2} follows from the definition of Maurer--Cartan elements. For \ref{tfae2} $\Leftrightarrow$ \ref{tfae3}, one calculates
\begin{align*}
P_\Phi (M)
&= P \big( M + [M, \Phi] + \tfrac12 [[M, \Phi], \Phi] + \tfrac16 [[[M, \Phi], \Phi], \Phi] + \dotsb \big) \\
&= P \big( [M, \Phi] + \tfrac12 [[M, \Phi], \Phi] \big)
\end{align*}
where
\begin{align}
\label{MPhi}
[M, \Phi] &= -\xi \circ_0 \phi \Oplus \xi \circ_1 \phi \Oplus -\xi \circ_0 \psi \Oplus \xi \circ_1 \psi \Oplus \phi \circ_0 \mu \Oplus \psi \circ_0 \nu
\intertext{and}
[[M, \Phi], \Phi] &= -(\xi \circ_0 \phi \circ_1 \phi + \xi \circ_1 \phi \circ_0 \phi) \Oplus -(\xi \circ_0 \phi \circ_1 \psi + \xi \circ_1 \psi \circ_0 \phi) \notag\\
&\qquad {} \Oplus -(\xi \circ_0 \psi \circ_1 \phi + \xi \circ_1 \phi \circ_0 \psi) \Oplus -(\xi \circ_0 \psi \circ_1 \psi + \xi \circ_1 \psi \circ_0 \psi) \notag\\
&= -2 \big( \xi \circ \phi^{\otimes 2} \Oplus \xi \circ (\phi \otimes \psi) \Oplus \xi \circ (\psi \otimes \phi) \Oplus \xi \circ \psi^{\otimes 2} \big) \label{MPhiPhi}
\end{align}
and higher commutators with $\Phi$ vanish. (Graphical illustrations of (\ref{MPhi}) and (\ref{MPhiPhi}) are given in (\ref{xPhi}) and (\ref{xPhiPhi}).)

We have that
\begin{align*}
P_{\Phi} M &= P \big( [M, \Phi] + \tfrac12 [[M, \Phi], \Phi] \big) \\
&= (\phi \circ \mu - \xi \circ \phi^{\otimes 2}) \oplus (\psi \circ \nu - \xi \circ \psi^{\otimes 2})
\end{align*}
which is zero if and only if
\begin{flalign*}
&& \phi \circ \mu &= \xi \circ \phi^{\otimes 2} && \\
&& \psi \circ \nu &= \xi \circ \psi^{\otimes 2} &&
\intertext{{\it i.e.}\ if and only if}
&& \phi (\mu (a, b)) &= \xi (\phi (a), \phi (b)) && \mathllap{\eqmakebox[lap][l]{$a, b \in A_U$}} \\
&& \psi (\nu (a, b)) &= \xi (\psi (a), \psi (b)) && \mathllap{\eqmakebox[lap][l]{$a, b \in A_V$}}
\end{flalign*}
so that $\phi$, respectively $\psi$, are morphisms compatible with the algebra structures on $A_U$ and $A_W$, respectively $A_V$ and $A_W$.
\end{proof}

We can now prove the following result.

\begin{theorem}
\label{maintheorem}
Let $(A_U, \mu) \toarg{\phi} (A_W, \xi) \leftarrowarg{\psi} (A_V, \nu)$ be a diagram of associative algebras over $U \tikzleftarrow W \tikzrightarrow V$. Let $\widetilde{\mathfrak g} \subset \mathfrak g$ be the subalgebra defined by
\[
\widetilde{\mathfrak g}\mkern2mu^n = \Hom (A_U^{\otimes n+1}, A_U) \oplus \Hom (A_V^{\otimes n+1}, A_V) \oplus \Hom (A_W^{\otimes n+1}, A_W).
\]
and let $\widetilde M = \widetilde\mu \oplus \widetilde\nu \oplus \widetilde\xi \in \widetilde{\mathfrak g}\mkern2mu^1$ and $\widetilde\Phi = \widetilde\phi \oplus \widetilde\psi \in \mathfrak a^0$ be arbitrary. Then
\[
\big( A_U, \mu + \widetilde\mu \big) \mathrel{\tikz[baseline] \path[-stealth,line width=.5pt] (0ex,0.65ex) edge node[above=-.5ex, overlay, font=\scriptsize] {$\phi + \widetilde\phi$} (6ex,.65ex);} \big( A_W, \xi + \widetilde\xi \,\big) \mathrel{\tikz[baseline] \path[stealth-,line width=.5pt] (0ex,0.65ex) edge node[above=-.5ex, overlay, font=\scriptsize] {$\psi + \widetilde\psi$} (6ex,.65ex);} \big( A_V, \nu + \widetilde\nu \big)
\]
is a diagram of associative algebras if and only if
\[
\widetilde M [1] \oplus \widetilde\Phi \in \MC \big( (\widetilde{\mathfrak g}[1] \oplus \mathfrak a)^{P_\Phi}_M \big).
\]
\end{theorem}

\begin{proof}
Calculate $(\widetilde M [1] \oplus \widetilde\Phi)^{\lmb n\rmb}$ (see \cite[\S 1.4]{fregierzambon1}) to see that $\exp_{\lmb\,\rmb} \Phi = 0$ corresponds precisely to
\begin{align*}
[M + \widetilde M, M + \widetilde M] &= 0 \\
P_{\Phi + \widetilde\Phi} (M + \widetilde M) &= 0.
\end{align*}
Now apply Propositions \ref{associative} and \ref{equivalent} \ref{tfae2}.
\end{proof}

\begin{definition}
\label{extendedgerstenhaber}
Let $x \circ_i^\Phi a = x \circ (\Phi^{\otimes i} \otimes a \otimes \Phi^{\otimes n-i})$ and then set $x \circ^\Phi a = \sum_{i=0}^n (-1)^{mi} x \circ_i^\Phi a$.
We then define
\[
[x, a]^\Phi = x \circ^\Phi a - (-1)^{\lvert x \rvert \lvert a \rvert} a \circ x
\]
which is essentially the Gerstenhaber bracket extended from the case of algebras to two algebras and morphisms between them.
\end{definition}

Recall from Remark \ref{vanishinghigherbrackets} that any L$_\infty [1]$ algebra $(\mathfrak g, \lmb \blank \rmb, \lmb \blank {,} \blank \rmb, \dotsc)$ defines a cochain complex $(\mathfrak g, d)$ for $d = \lmb \blank \rmb$.

\begin{proposition}
\label{unary}
The cochain complex underlying the L$_\infty [1]$ algebra $(\widetilde{\mathfrak g}[1] \oplus \mathfrak a)^{P_\Phi}_M$ coincides with the (truncated) Gerstenhaber--Schack complex. In particular,
\[
\lmb x[1] \oplus a \rmb = d_{\mathrm{GS}} (x[1] \oplus a) = (-1)^{\lvert x \rvert + 1} (d_{\mathrm H} x)[1] \oplus d_\Delta x + (-1)^{\lvert a \rvert} d_{\mathrm H} a \; \in \widetilde{\mathfrak g}[1] \oplus \mathfrak a
\]
for homogeneous elements $x \in \widetilde{\mathfrak g}$ and $a \in \mathfrak a$.
\end{proposition}

\begin{proof}
That $\widetilde{\mathfrak g}[1] \oplus \mathfrak a$ coincides with $\mathrm C^\hdot_{\mathrm{GS}} (\mathbb A)$ as graded $\mathbb k$-vector space is clear from the definitions of $\mathfrak g$, $\widetilde{\mathfrak g}$ and $\mathfrak a$, see (\ref{ga}) and Theorem \ref{maintheorem}.

To show that $\lmb \blank \rmb = d_{\mathrm{GS}}$ we compute the unary bracket in Theorem \ref{voronov} applied to $(\widetilde{\mathfrak g}[1] \oplus \mathfrak a)^{P_\Phi}_M$. For $x \in \widetilde{\mathfrak g}^n$ and $a \in \mathfrak a^m$, Theorem \ref{voronov} gives the following formula
\[
\lmb x[1] \oplus a \rmb = -[M, x][1] \oplus P_\Phi (x + [M, a])
\]
where
\begin{equation}
\label{MPhixa}
\begin{aligned}
   M &= \eqmakebox[munuxiphipsi]{$(\mu, \nu, \xi)$} \in \widetilde{\mathfrak g}^1 \qquad&
   x &= \eqmakebox[xa]{$(x_U, x_V, x_W)$} \in \widetilde{\mathfrak g}^n \\
\Phi &= \eqmakebox[munuxiphipsi]{$(\phi, \psi)$} \in \mathfrak a^0 &
   a &= \eqmakebox[xa]{$(a_{W \smallto U}, a_{W \smallto V})$} \in \mathfrak a^m
\end{aligned}
\end{equation}
and
\begin{equation*}
P_\Phi (\blank) = P \circ \exp [\blank, \Phi] = P (\blank) + P [\blank, \Phi] + \tfrac12 P [[\blank, \Phi], \Phi] + \dotsb
\end{equation*}
We give a visual proof, denoting inputs and outputs in $A_U, A_V, A_W$ by
\hair
$\begin{tikzpicture}[baseline=-3,x=.75em,y=.75em] \draw[line width=.65pt] node[draw,shape=rectangle,scale=.6] {}; \end{tikzpicture}$\hair, 
$\begin{tikzpicture}[baseline=-3,x=.75em,y=.75em] \draw[line width=.65pt] node[draw,shape=diamond,scale=.43] {}; \end{tikzpicture}$\hair, 
$\begin{tikzpicture}[baseline=-3,x=.75em,y=.75em] \draw[line width=.65pt] node[draw,shape=circle,scale=.5] {}; \end{tikzpicture}$\hair, 
respectively, to write (\ref{MPhixa}) in operadic notation as
\begin{alignat*}{5}
M &=
\Big(
\,
\begin{tikzpicture}[baseline=7,x=.75em,y=.75em]
\draw[line width=.6pt] (-.5,0) node[draw,shape=rectangle,scale=.6] (L) {};
\draw[line width=.6pt] (.5,0)  node[draw,shape=rectangle,scale=.6] (R) {};
\draw[line width=.6pt] (0,2) node[draw,shape=rectangle,scale=.6] (T) {};
\draw[line width=.6pt] (L) -- ++(0,.75) -- ++(.5,.5) -- (T);
\draw[line width=.6pt] (R) -- ++(0,.75) -- ++(-.5,.5);
\end{tikzpicture}
\,
\raisebox{-1.5ex}{,}
\,
\begin{tikzpicture}[baseline=7,x=.75em,y=.75em]
\draw[line width=.6pt] (-.5,0) node[draw,shape=diamond,scale=.43] (L) {};
\draw[line width=.6pt] (.5,0)  node[draw,shape=diamond,scale=.43] (R) {};
\draw[line width=.6pt] (0,2) node[draw,shape=diamond,scale=.43] (T) {};
\draw[line width=.6pt] (L) -- ++(0,.75) -- ++(.5,.5) -- (T);
\draw[line width=.6pt] (R) -- ++(0,.75) -- ++(-.5,.5);
\end{tikzpicture}
\,
\raisebox{-1.5ex}{,}
\,
\begin{tikzpicture}[baseline=7,x=.75em,y=.75em]
\draw[line width=.6pt] (-.5,0) node[draw,shape=circle,scale=.5] (L) {};
\draw[line width=.6pt] (.5,0)  node[draw,shape=circle,scale=.5] (R) {};
\draw[line width=.6pt] (0,2) node[draw,shape=circle,scale=.5] (T) {};
\draw[line width=.6pt] (L) -- ++(0,.75) -- ++(.5,.5) -- (T);
\draw[line width=.6pt] (R) -- ++(0,.75) -- ++(-.5,.5);
\end{tikzpicture}
\,
\Big) & {} \in {} & \widetilde{\mathfrak g}^1 \qquad\quad
& x &= 
\bigg(
\,
\begin{tikzpicture}[baseline=9,x=.75em,y=.75em]
\draw[line width=.6pt] (-1,0) node[draw,shape=rectangle,scale=.6] (L) {};
\draw[line width=.6pt] (1,0)  node[draw,shape=rectangle,scale=.6] (R) {};
\draw[line width=.6pt] (0,2.5) node[draw,shape=rectangle,scale=.6] (T) {};
\draw (0,0) node[font=\scriptsize] {...};
\draw[line width=.6pt] (L) -- ++(0,.75) -- ++(1,1) -- (T);
\draw[line width=.6pt] (R) -- ++(0,.75) -- ++(-1,1);
\end{tikzpicture}
\,
\raisebox{-2ex}{,}
\,
\begin{tikzpicture}[baseline=9,x=.75em,y=.75em]
\draw[line width=.6pt] (-1,0) node[draw,shape=diamond,scale=.43] (L) {};
\draw[line width=.6pt] (1,0)  node[draw,shape=diamond,scale=.43] (R) {};
\draw[line width=.6pt] (0,2.5) node[draw,shape=diamond,scale=.43] (T) {};
\draw (0,0) node[font=\scriptsize] {...};
\draw[line width=.6pt] (L) -- ++(0,.75) -- ++(1,1) -- (T);
\draw[line width=.6pt] (R) -- ++(0,.75) -- ++(-1,1);
\end{tikzpicture}
\,
\raisebox{-2ex}{,}
\,
\begin{tikzpicture}[baseline=9,x=.75em,y=.75em]
\draw[line width=.6pt] (-1,0) node[draw,shape=circle,scale=.5] (L) {};
\draw[line width=.6pt] (1,0)  node[draw,shape=circle,scale=.5] (R) {};
\draw[line width=.6pt] (0,2.5) node[draw,shape=circle,scale=.5] (T) {};
\draw (0,0) node[font=\scriptsize] {...};
\draw[line width=.6pt] (L) -- ++(0,.75) -- ++(1,1) -- (T);
\draw[line width=.6pt] (R) -- ++(0,.75) -- ++(-1,1);
\end{tikzpicture}
\,
\bigg) & {} \in {} & \widetilde{\mathfrak g}^n
\\
\Phi &=
\Big(
\,
\begin{tikzpicture}[baseline=4.5,x=.75em,y=.75em]
\draw[line width=.6pt] (0,0) node[draw,shape=rectangle,scale=.6] (B) {};
\draw[line width=.6pt] (0,1.5) node[draw,shape=circle,scale=.5] (T) {};
\draw[line width=.6pt] (B) -- (T);
\end{tikzpicture}
\,
\raisebox{-1ex}{,}
\,
\begin{tikzpicture}[baseline=4.5,x=.75em,y=.75em]
\draw[line width=.6pt] (0,0) node[draw,shape=diamond,scale=.43] (B) {};
\draw[line width=.6pt] (0,1.5) node[draw,shape=circle,scale=.5] (T) {};
\draw[line width=.6pt] (B) -- (T);
\end{tikzpicture}
\,
\Big) & {} \in {} & \mathfrak a^0 
& a &=
\bigg(
\,
\begin{tikzpicture}[baseline=9,x=.75em,y=.75em]
\draw[line width=.6pt] (-1,0) node[draw,shape=rectangle,scale=.6] (L) {};
\draw[line width=.6pt] (1,0)  node[draw,shape=rectangle,scale=.6] (R) {};
\draw[line width=.6pt] (0,2.5) node[draw,shape=circle,scale=.5] (T) {};
\draw (0,0) node[font=\scriptsize] {...};
\draw[line width=.6pt] (L) -- ++(0,.75) -- ++(1,1) -- (T);
\draw[line width=.6pt] (R) -- ++(0,.75) -- ++(-1,1);
\end{tikzpicture}
\,
\raisebox{-2ex}{,}
\,
\begin{tikzpicture}[baseline=9,x=.75em,y=.75em]
\draw[line width=.6pt] (-1,0) node[draw,shape=diamond,scale=.43] (L) {};
\draw[line width=.6pt] (1,0)  node[draw,shape=diamond,scale=.43] (R) {};
\draw[line width=.6pt] (0,2.5) node[draw,shape=circle,scale=.5] (T) {};
\draw (0,0) node[font=\scriptsize] {...};
\draw[line width=.6pt] (L) -- ++(0,.75) -- ++(1,1) -- (T);
\draw[line width=.6pt] (R) -- ++(0,.75) -- ++(-1,1);
\end{tikzpicture}
\,
\bigg) & {} \in {} & \mathfrak a^m
\end{alignat*}
(Inputs are at the bottom and outputs at the top of the diagram.) We can now calculate the Gerstenhaber brackets in $\mathfrak g$.

\paragraph{\bf Computation of $\lmb x[1] \rmb = -[M, x][1] \oplus P_\Phi (x)$.} \mbox{}

\noindent That $[M, x] = d_{\mathrm H} (x)$ is clear from the definition of $d_{\mathrm H}$. (The sign $(-1)^{\lvert x \rvert + 1}$ appears as part of the total differential.)

To calculate $P_\Phi (x) = P (x + [x, \Phi] + \tfrac12 [[x, \Phi], \Phi] + \dotsb)$ we calculate
\begin{align}
[x, \Phi] &=
\bigoplus_{i=0}^n
\begin{tikzpicture}[baseline=5,x=.75em,y=.75em]
\draw[line width=.6pt] (-2,0)   node[draw,shape=circle,scale=.5] (L) {};
\draw[line width=.6pt] (2,0)    node[draw,shape=circle,scale=.5] (R) {};
\draw[line width=.6pt] (0,0)    node[draw,shape=circle,scale=.5] (M) {};
\draw[line width=.6pt] (0,2.5)  node[draw,shape=circle,scale=.5] (T) {};
\draw[line width=.6pt] (0,-1.5) node[draw,shape=rectangle,scale=.6] (B) {};
\draw (-1,0) node[font=\scriptsize] {...};
\draw (1,0) node[font=\scriptsize] {...};
\draw[line width=.6pt] (L) -- ++(0,.75) -- ++(2,1) -- (T);
\draw[line width=.6pt] (R) -- ++(0,.75) -- ++(-2,1);
\draw[line width=.6pt] (M) -- (T);
\draw[line width=.6pt] (B) -- (M);
\draw (0,.75) node[left=-.2em,font=\scriptsize] {$i$};
\end{tikzpicture}
\;\oplus\;
\bigoplus_{i=0}^n
\begin{tikzpicture}[baseline=5,x=.75em,y=.75em]
\draw[line width=.6pt] (-2,0)   node[draw,shape=circle,scale=.5] (L) {};
\draw[line width=.6pt] (2,0)    node[draw,shape=circle,scale=.5] (R) {};
\draw[line width=.6pt] (0,0)    node[draw,shape=circle,scale=.5] (M) {};
\draw[line width=.6pt] (0,2.5)  node[draw,shape=circle,scale=.5] (T) {};
\draw[line width=.6pt] (0,-1.5) node[draw,shape=diamond,scale=.43] (B) {};
\draw (-1,0) node[font=\scriptsize] {...};
\draw (1,0) node[font=\scriptsize] {...};
\draw[line width=.6pt] (L) -- ++(0,.75) -- ++(2,1) -- (T);
\draw[line width=.6pt] (R) -- ++(0,.75) -- ++(-2,1);
\draw[line width=.6pt] (M) -- (T);
\draw[line width=.6pt] (B) -- (M);
\draw (0,.75) node[left=-.2em,font=\scriptsize] {$i$};
\end{tikzpicture}
\;\oplus\;
(-1) 
\begin{tikzpicture}[baseline=15,x=.75em,y=.75em]
\draw[line width=.6pt] (-1,0) node[draw,shape=rectangle,scale=.6] (L) {};
\draw[line width=.6pt] (1,0)  node[draw,shape=rectangle,scale=.6] (R) {};
\draw[line width=.6pt] (0,2.5) node[draw,shape=rectangle,scale=.6] (T) {};
\draw[line width=.6pt] (0,4) node[draw,shape=circle,scale=.5] (TT) {};
\draw (0,0) node[font=\scriptsize] {...};
\draw[line width=.6pt] (L) -- ++(0,.75) -- ++(1,1) -- (T);
\draw[line width=.6pt] (R) -- ++(0,.75) -- ++(-1,1);
\draw[line width=.6pt] (T) -- (TT);
\end{tikzpicture}
\;\oplus\;
(-1)
\begin{tikzpicture}[baseline=15,x=.75em,y=.75em]
\draw[line width=.6pt] (-1,0)  node[draw,shape=diamond,scale=.43] (L) {};
\draw[line width=.6pt] (1,0)   node[draw,shape=diamond,scale=.43] (R) {};
\draw[line width=.6pt] (0,2.5) node[draw,shape=diamond,scale=.43] (T) {};
\draw[line width=.6pt] (0,4)   node[draw,shape=circle,scale=.5] (TT) {};
\draw (0,0) node[font=\scriptsize] {...};
\draw[line width=.6pt] (L) -- ++(0,.75) -- ++(1,1) -- (T);
\draw[line width=.6pt] (R) -- ++(0,.75) -- ++(-1,1);
\draw[line width=.6pt] (T) -- (TT);
\end{tikzpicture}
\label{xPhi}
\\
\tfrac12 [[x, \Phi], \Phi] &=
\bigoplus_{i < j}
\begin{tikzpicture}[baseline=5,x=.75em,y=.75em]
\draw[line width=.6pt] (-3,0)    node[draw,shape=circle,scale=.5] (L) {};
\draw[line width=.6pt] (3,0)     node[draw,shape=circle,scale=.5] (R) {};
\draw[line width=.6pt] (-1,0)    node[draw,shape=circle,scale=.5] (ML) {};
\draw[line width=.6pt] (1,0)     node[draw,shape=circle,scale=.5] (MR) {};
\draw[line width=.6pt] (0,2.5)   node[draw,shape=circle,scale=.5] (T) {};
\draw[line width=.6pt] (-1,-1.5) node[draw,shape=rectangle,scale=.6] (BL) {};
\draw[line width=.6pt] (1,-1.5)  node[draw,shape=rectangle,scale=.6] (BR) {};
\draw (-2,0) node[font=\scriptsize] {...};
\draw (0,0) node[font=\scriptsize] {...};
\draw (2,0) node[font=\scriptsize] {...};
\draw[line width=.6pt] (L) -- ++(0,.75) -- ++(3,1) -- (T);
\draw[line width=.6pt] (R) -- ++(0,.75) -- ++(-3,1);
\draw[line width=.6pt] (ML) -- ++(0,.75) -- ++(1,1);
\draw[line width=.6pt] (MR) -- ++(0,.75) -- ++(-1,1);
\draw[line width=.6pt] (BL) -- (ML);
\draw[line width=.6pt] (BR) -- (MR);
\draw (0,.75) node[left=-.1em,font=\scriptsize] {$i$};
\draw (0,.65) node[right=-.24em,font=\scriptsize] {$j$};
\end{tikzpicture}
\;\oplus\;
\bigoplus_{i < j}
\begin{tikzpicture}[baseline=5,x=.75em,y=.75em]
\draw[line width=.6pt] (-3,0)    node[draw,shape=circle,scale=.5] (L) {};
\draw[line width=.6pt] (3,0)     node[draw,shape=circle,scale=.5] (R) {};
\draw[line width=.6pt] (-1,0)    node[draw,shape=circle,scale=.5] (ML) {};
\draw[line width=.6pt] (1,0)     node[draw,shape=circle,scale=.5] (MR) {};
\draw[line width=.6pt] (0,2.5)   node[draw,shape=circle,scale=.5] (T) {};
\draw[line width=.6pt] (-1,-1.5) node[draw,shape=diamond,scale=.43] (BL) {};
\draw[line width=.6pt] (1,-1.5)  node[draw,shape=diamond,scale=.43] (BR) {};
\draw (-2,0) node[font=\scriptsize] {...};
\draw (0,0) node[font=\scriptsize] {...};
\draw (2,0) node[font=\scriptsize] {...};
\draw[line width=.6pt] (L) -- ++(0,.75) -- ++(3,1) -- (T);
\draw[line width=.6pt] (R) -- ++(0,.75) -- ++(-3,1);
\draw[line width=.6pt] (ML) -- ++(0,.75) -- ++(1,1);
\draw[line width=.6pt] (MR) -- ++(0,.75) -- ++(-1,1);
\draw[line width=.6pt] (BL) -- (ML);
\draw[line width=.6pt] (BR) -- (MR);
\draw (0,.75) node[left=-.1em,font=\scriptsize] {$i$};
\draw (0,.65) node[right=-.24em,font=\scriptsize] {$j$};
\end{tikzpicture}
\notag
\\
&\phantom{={}}
\;\oplus\;
\bigoplus_{i < j}
\begin{tikzpicture}[baseline=5,x=.75em,y=.75em]
\draw[line width=.6pt] (-3,0)    node[draw,shape=circle,scale=.5] (L) {};
\draw[line width=.6pt] (3,0)     node[draw,shape=circle,scale=.5] (R) {};
\draw[line width=.6pt] (-1,0)    node[draw,shape=circle,scale=.5] (ML) {};
\draw[line width=.6pt] (1,0)     node[draw,shape=circle,scale=.5] (MR) {};
\draw[line width=.6pt] (0,2.5)   node[draw,shape=circle,scale=.5] (T) {};
\draw[line width=.6pt] (-1,-1.5) node[draw,shape=rectangle,scale=.6] (BL) {};
\draw[line width=.6pt] (1,-1.5)  node[draw,shape=diamond,scale=.43] (BR) {};
\draw (-2,0) node[font=\scriptsize] {...};
\draw (0,0) node[font=\scriptsize] {...};
\draw (2,0) node[font=\scriptsize] {...};
\draw[line width=.6pt] (L) -- ++(0,.75) -- ++(3,1) -- (T);
\draw[line width=.6pt] (R) -- ++(0,.75) -- ++(-3,1);
\draw[line width=.6pt] (ML) -- ++(0,.75) -- ++(1,1);
\draw[line width=.6pt] (MR) -- ++(0,.75) -- ++(-1,1);
\draw[line width=.6pt] (BL) -- (ML);
\draw[line width=.6pt] (BR) -- (MR);
\draw (0,.75) node[left=-.1em,font=\scriptsize] {$i$};
\draw (0,.65) node[right=-.24em,font=\scriptsize] {$j$};
\end{tikzpicture}
\;\oplus\;
\bigoplus_{i < j}
\begin{tikzpicture}[baseline=5,x=.75em,y=.75em]
\draw[line width=.6pt] (-3,0)    node[draw,shape=circle,scale=.5] (L) {};
\draw[line width=.6pt] (3,0)     node[draw,shape=circle,scale=.5] (R) {};
\draw[line width=.6pt] (-1,0)    node[draw,shape=circle,scale=.5] (ML) {};
\draw[line width=.6pt] (1,0)     node[draw,shape=circle,scale=.5] (MR) {};
\draw[line width=.6pt] (0,2.5)   node[draw,shape=circle,scale=.5] (T) {};
\draw[line width=.6pt] (-1,-1.5) node[draw,shape=diamond,scale=.43] (BL) {};
\draw[line width=.6pt] (1,-1.5)  node[draw,shape=rectangle,scale=.6] (BR) {};
\draw (-2,0) node[font=\scriptsize] {...};
\draw (0,0) node[font=\scriptsize] {...};
\draw (2,0) node[font=\scriptsize] {...};
\draw[line width=.6pt] (L) -- ++(0,.75) -- ++(3,1) -- (T);
\draw[line width=.6pt] (R) -- ++(0,.75) -- ++(-3,1);
\draw[line width=.6pt] (ML) -- ++(0,.75) -- ++(1,1);
\draw[line width=.6pt] (MR) -- ++(0,.75) -- ++(-1,1);
\draw[line width=.6pt] (BL) -- (ML);
\draw[line width=.6pt] (BR) -- (MR);
\draw (0,.75) node[left=-.1em,font=\scriptsize] {$i$};
\draw (0,.65) node[right=-.24em,font=\scriptsize] {$j$};
\end{tikzpicture}
\label{xPhiPhi}
\end{align}
Since $x \notin \mathfrak a$, $P (x) = 0$. In the expression of $[x, \Phi]$ only the last two terms are in $\mathfrak a$, so
\[
P [x, \Phi] = 
(-1) 
\begin{tikzpicture}[baseline=15,x=.75em,y=.75em]
\draw[line width=.6pt] (-1,0) node[draw,shape=rectangle,scale=.6] (L) {};
\draw[line width=.6pt] (1,0)  node[draw,shape=rectangle,scale=.6] (R) {};
\draw[line width=.6pt] (0,2.5) node[draw,shape=rectangle,scale=.6] (T) {};
\draw[line width=.6pt] (0,4) node[draw,shape=circle,scale=.5] (TT) {};
\draw (0,0) node[font=\scriptsize] {...};
\draw[line width=.6pt] (L) -- ++(0,.75) -- ++(1,1) -- (T);
\draw[line width=.6pt] (R) -- ++(0,.75) -- ++(-1,1);
\draw[line width=.6pt] (T) -- (TT);
\end{tikzpicture}
\;\oplus\;
(-1)
\begin{tikzpicture}[baseline=15,x=.75em,y=.75em]
\draw[line width=.6pt] (-1,0)  node[draw,shape=diamond,scale=.43] (L) {};
\draw[line width=.6pt] (1,0)   node[draw,shape=diamond,scale=.43] (R) {};
\draw[line width=.6pt] (0,2.5) node[draw,shape=diamond,scale=.43] (T) {};
\draw[line width=.6pt] (0,4)   node[draw,shape=circle,scale=.5] (TT) {};
\draw (0,0) node[font=\scriptsize] {...};
\draw[line width=.6pt] (L) -- ++(0,.75) -- ++(1,1) -- (T);
\draw[line width=.6pt] (R) -- ++(0,.75) -- ++(-1,1);
\draw[line width=.6pt] (T) -- (TT);
\end{tikzpicture}
= - \Phi \circ x
\]
In $[[x, \Phi], \Phi]$ none of the terms are in $\mathfrak a$, so $P [[x, \Phi], \Phi] = 0$.

Similarly for $[...[[x, \Phi], \Phi], \dotsc, \Phi]$, the only terms surviving the projection to $\mathfrak a$ are terms where $\Phi$ has been precomposed in all inputs of $x$, so that
\[
\frac1{n!}P [...[[x, \underbrace{\Phi], \Phi], \dotsc, \Phi}_{n+1}] = 
\begin{tikzpicture}[baseline=5,x=.75em,y=.75em]
\draw[line width=.6pt] (-1,0)   node[draw,shape=circle,scale=.5] (L) {};
\draw[line width=.6pt] (1,0)    node[draw,shape=circle,scale=.5] (R) {};
\draw[line width=.6pt] (0,2.5)  node[draw,shape=circle,scale=.5] (T) {};
\draw[line width=.6pt] (-1,-1.5) node[draw,shape=rectangle,scale=.6] (BL) {};
\draw[line width=.6pt] (1,-1.5) node[draw,shape=rectangle,scale=.6] (BR) {};
\draw (0,0) node[font=\scriptsize] {...};
\draw (0,-1.5) node[font=\scriptsize] {...};
\draw[line width=.6pt] (L) -- ++(0,.75) -- ++(1,1) -- (T);
\draw[line width=.6pt] (R) -- ++(0,.75) -- ++(-1,1);
\draw[line width=.6pt] (BL) -- (L);
\draw[line width=.6pt] (BR) -- (R);
\end{tikzpicture}
\;\oplus\;
\begin{tikzpicture}[baseline=5,x=.75em,y=.75em]
\draw[line width=.6pt] (-1,0)   node[draw,shape=circle,scale=.5] (L) {};
\draw[line width=.6pt] (1,0)    node[draw,shape=circle,scale=.5] (R) {};
\draw[line width=.6pt] (0,2.5)  node[draw,shape=circle,scale=.5] (T) {};
\draw[line width=.6pt] (-1,-1.5) node[draw,shape=diamond,scale=.43] (BL) {};
\draw[line width=.6pt] (1,-1.5) node[draw,shape=diamond,scale=.43] (BR) {};
\draw (0,0) node[font=\scriptsize] {...};
\draw (0,-1.5) node[font=\scriptsize] {...};
\draw[line width=.6pt] (L) -- ++(0,.75) -- ++(1,1) -- (T);
\draw[line width=.6pt] (R) -- ++(0,.75) -- ++(-1,1);
\draw[line width=.6pt] (BL) -- (L);
\draw[line width=.6pt] (BR) -- (R);
\end{tikzpicture}
= x \circ \Phi^{\otimes n+1}
\]
We thus have
\[
P_\Phi (x) = x \circ \Phi^{\otimes n+1} - \Phi \circ x = d_\Delta (x),
\]
where $d_\Delta$ is the simplicial differential of Definition \ref{definitiongerstenhaberschack}, whose formula is given in Definition \ref{definitionsimplicial}.

\paragraph{Computation of $P_\Phi ([M, a])$.} \mbox{}

\noindent To calculate $P_\Phi (a) = P ([M, a] + [[M, a], \Phi] + \tfrac12 [[[M, a], \Phi], \Phi] + \dotsb)$ we calculate
\begin{align*}
[M, a] &= 
\begin{tikzpicture}[baseline=15,x=.75em,y=.75em]
\draw[line width=.6pt] (-1,0) node[draw,shape=rectangle,scale=.6] (L) {};
\draw[line width=.6pt] (1,0)  node[draw,shape=rectangle,scale=.6] (R) {};
\draw[line width=.6pt] (0,2.5) node[draw,shape=circle,scale=.5] (T) {};
\draw[line width=.6pt] (1,2.5) node[draw,shape=circle,scale=.5] (TR) {};
\draw[line width=.6pt] (.5,4.5) node[draw,shape=circle,scale=.5] (TT) {};
\draw (0,0) node[font=\scriptsize] {...};
\draw[line width=.6pt] (L) -- ++(0,.75) -- ++(1,1) -- (T);
\draw[line width=.6pt] (R) -- ++(0,.75) -- ++(-1,1);
\draw[line width=.6pt] (TR) -- ++(0,.75) -- ++(-.5,.5) -- (TT);
\draw[line width=.6pt] (T) -- ++(0,.75) -- ++(.5,.5);
\end{tikzpicture}
\;\oplus
(-1)^m\;
\begin{tikzpicture}[baseline=15,x=.75em,y=.75em]
\draw[line width=.6pt] (-1,0) node[draw,shape=rectangle,scale=.6] (L) {};
\draw[line width=.6pt] (1,0)  node[draw,shape=rectangle,scale=.6] (R) {};
\draw[line width=.6pt] (0,2.5) node[draw,shape=circle,scale=.5] (T) {};
\draw[line width=.6pt] (-1,2.5) node[draw,shape=circle,scale=.5] (TL) {};
\draw[line width=.6pt] (-.5,4.5) node[draw,shape=circle,scale=.5] (TT) {};
\draw (0,0) node[font=\scriptsize] {...};
\draw[line width=.6pt] (L) -- ++(0,.75) -- ++(1,1) -- (T);
\draw[line width=.6pt] (R) -- ++(0,.75) -- ++(-1,1);
\draw[line width=.6pt] (TL) -- ++(0,.75) -- ++(.5,.5) -- (TT);
\draw[line width=.6pt] (T) -- ++(0,.75) -- ++(-.5,.5);
\end{tikzpicture}
\;\oplus\;
\begin{tikzpicture}[baseline=15,x=.75em,y=.75em]
\draw[line width=.6pt] (-1,0) node[draw,shape=diamond,scale=.43] (L) {};
\draw[line width=.6pt] (1,0)  node[draw,shape=diamond,scale=.43] (R) {};
\draw[line width=.6pt] (0,2.5) node[draw,shape=circle,scale=.5] (T) {};
\draw[line width=.6pt] (1,2.5) node[draw,shape=circle,scale=.5] (TR) {};
\draw[line width=.6pt] (.5,4.5) node[draw,shape=circle,scale=.5] (TT) {};
\draw (0,0) node[font=\scriptsize] {...};
\draw[line width=.6pt] (L) -- ++(0,.75) -- ++(1,1) -- (T);
\draw[line width=.6pt] (R) -- ++(0,.75) -- ++(-1,1);
\draw[line width=.6pt] (TR) -- ++(0,.75) -- ++(-.5,.5) -- (TT);
\draw[line width=.6pt] (T) -- ++(0,.75) -- ++(.5,.5);
\end{tikzpicture}
\;\oplus
(-1)^m\;
\begin{tikzpicture}[baseline=15,x=.75em,y=.75em]
\draw[line width=.6pt] (-1,0) node[draw,shape=diamond,scale=.43] (L) {};
\draw[line width=.6pt] (1,0)  node[draw,shape=diamond,scale=.43] (R) {};
\draw[line width=.6pt] (0,2.5) node[draw,shape=circle,scale=.5] (T) {};
\draw[line width=.6pt] (-1,2.5) node[draw,shape=circle,scale=.5] (TL) {};
\draw[line width=.6pt] (-.5,4.5) node[draw,shape=circle,scale=.5] (TT) {};
\draw (0,0) node[font=\scriptsize] {...};
\draw[line width=.6pt] (L) -- ++(0,.75) -- ++(1,1) -- (T);
\draw[line width=.6pt] (R) -- ++(0,.75) -- ++(-1,1);
\draw[line width=.6pt] (TL) -- ++(0,.75) -- ++(.5,.5) -- (TT);
\draw[line width=.6pt] (T) -- ++(0,.75) -- ++(-.5,.5);
\end{tikzpicture}
\\
& \phantom{={}} \;\oplus
(-1)^{m+1}
\sum_{i = 0}^m
(-1)^i\;
\begin{tikzpicture}[baseline=-3,x=.75em,y=.75em]
\draw[line width=.6pt] (-2,0) node[draw,shape=rectangle,scale=.6] (L) {};
\draw[line width=.6pt] (2,0)  node[draw,shape=rectangle,scale=.6] (R) {};
\draw[line width=.6pt] (0,0)  node[draw,shape=rectangle,scale=.6] (M) {};
\draw[line width=.6pt] (0,2.5) node[draw,shape=circle,scale=.5] (T) {};
\draw[line width=.6pt] (-.5,-2) node[draw,shape=rectangle,scale=.6] (BL) {};
\draw[line width=.6pt] (.5,-2) node[draw,shape=rectangle,scale=.6] (BR) {};
\draw (-1,0) node[font=\scriptsize] {...};
\draw (1,0) node[font=\scriptsize] {...};
\draw[line width=.6pt] (L) -- ++(0,.75) -- ++(2,1) -- (T);
\draw[line width=.6pt] (R) -- ++(0,.75) -- ++(-2,1);
\draw[line width=.6pt] (BL) -- ++(0,.75) -- ++(.5,.5) -- (M);
\draw[line width=.6pt] (BR) -- ++(0,.75) -- ++(-.5,.5);
\draw[line width=.6pt] (M) -- ++(0,.75) -- ++(0,1);
\end{tikzpicture}
\;\oplus
(-1)^{m+1}
\sum_{i = 0}^m
(-1)^i \;
\begin{tikzpicture}[baseline=-3,x=.75em,y=.75em]
\draw[line width=.6pt] (-2,0) node[draw,shape=diamond,scale=.43] (L) {};
\draw[line width=.6pt] (2,0)  node[draw,shape=diamond,scale=.43] (R) {};
\draw[line width=.6pt] (0,0)  node[draw,shape=diamond,scale=.43] (M) {};
\draw[line width=.6pt] (0,2.5) node[draw,shape=circle,scale=.5] (T) {};
\draw[line width=.6pt] (-.5,-2) node[draw,shape=diamond,scale=.43] (BL) {};
\draw[line width=.6pt] (.5,-2) node[draw,shape=diamond,scale=.43] (BR) {};
\draw (-1,0) node[font=\scriptsize] {...};
\draw (1,0) node[font=\scriptsize] {...};
\draw[line width=.6pt] (L) -- ++(0,.75) -- ++(2,1) -- (T);
\draw[line width=.6pt] (R) -- ++(0,.75) -- ++(-2,1);
\draw[line width=.6pt] (BL) -- ++(0,.75) -- ++(.5,.5) -- (M);
\draw[line width=.6pt] (BR) -- ++(0,.75) -- ++(-.5,.5);
\draw[line width=.6pt] (M) -- ++(0,.75) -- ++(0,1);
\end{tikzpicture}
\\
[[M, a], \Phi] &= 
\begin{tikzpicture}[baseline=15,x=.75em,y=.75em]
\draw[line width=.6pt] (-1.5,-.5) node[draw,shape=rectangle,scale=.6] (L) {};
\draw[line width=.6pt] (.5,-.5)  node[draw,shape=rectangle,scale=.6] (R) {};
\draw[line width=.6pt] (1.5,-.5)  node[draw,shape=rectangle,scale=.6] (RR) {};
\draw[line width=.6pt] (-.5,2) node[draw,shape=circle,scale=.5] (T) {};
\draw[line width=.6pt] (1.5,2) node[draw,shape=circle,scale=.5] (TR) {};
\draw[line width=.6pt] (.5,4.5) node[draw,shape=circle,scale=.5] (TT) {};
\draw (-.5,-.5) node[font=\scriptsize] {...};
\draw[line width=.6pt] (L) -- ++(0,.75) -- ++(1,1) -- (T);
\draw[line width=.6pt] (R) -- ++(0,.75) -- ++(-1,1);
\draw[line width=.6pt] (TR) -- ++(0,.75) -- ++(-1,1) -- (TT);
\draw[line width=.6pt] (T) -- ++(0,.75) -- ++(1,1);
\draw[line width=.6pt] (RR) -- (TR);
\end{tikzpicture}
\;+
(-1)^m \;
\begin{tikzpicture}[baseline=15,x=.75em,y=.75em]
\draw[line width=.6pt] (1.5,-.5) node[draw,shape=rectangle,scale=.6] (L) {};
\draw[line width=.6pt] (-.5,-.5)  node[draw,shape=rectangle,scale=.6] (R) {};
\draw[line width=.6pt] (-1.5,-.5)  node[draw,shape=rectangle,scale=.6] (RR) {};
\draw[line width=.6pt] (.5,2) node[draw,shape=circle,scale=.5] (T) {};
\draw[line width=.6pt] (-1.5,2) node[draw,shape=circle,scale=.5] (TR) {};
\draw[line width=.6pt] (-.5,4.5) node[draw,shape=circle,scale=.5] (TT) {};
\draw (.5,-.5) node[font=\scriptsize] {...};
\draw[line width=.6pt] (L) -- ++(0,.75) -- ++(-1,1) -- (T);
\draw[line width=.6pt] (R) -- ++(0,.75) -- ++(1,1);
\draw[line width=.6pt] (TR) -- ++(0,.75) -- ++(1,1) -- (TT);
\draw[line width=.6pt] (T) -- ++(0,.75) -- ++(-1,1);
\draw[line width=.6pt] (RR) -- (TR);
\end{tikzpicture}
\;\oplus\;
\begin{tikzpicture}[baseline=15,x=.75em,y=.75em]
\draw[line width=.6pt] (-1.5,-.5) node[draw,shape=diamond,scale=.43] (L) {};
\draw[line width=.6pt] (.5,-.5)  node[draw,shape=diamond,scale=.43] (R) {};
\draw[line width=.6pt] (1.5,-.5)  node[draw,shape=diamond,scale=.43] (RR) {};
\draw[line width=.6pt] (-.5,2) node[draw,shape=circle,scale=.5] (T) {};
\draw[line width=.6pt] (1.5,2) node[draw,shape=circle,scale=.5] (TR) {};
\draw[line width=.6pt] (.5,4.5) node[draw,shape=circle,scale=.5] (TT) {};
\draw (-.5,-.5) node[font=\scriptsize] {...};
\draw[line width=.6pt] (L) -- ++(0,.75) -- ++(1,1) -- (T);
\draw[line width=.6pt] (R) -- ++(0,.75) -- ++(-1,1);
\draw[line width=.6pt] (TR) -- ++(0,.75) -- ++(-1,1) -- (TT);
\draw[line width=.6pt] (T) -- ++(0,.75) -- ++(1,1);
\draw[line width=.6pt] (RR) -- (TR);
\end{tikzpicture}
\;+
(-1)^m \;
\begin{tikzpicture}[baseline=15,x=.75em,y=.75em]
\draw[line width=.6pt] (1.5,-.5) node[draw,shape=diamond,scale=.43] (L) {};
\draw[line width=.6pt] (-.5,-.5)  node[draw,shape=diamond,scale=.43] (R) {};
\draw[line width=.6pt] (-1.5,-.5)  node[draw,shape=diamond,scale=.43] (RR) {};
\draw[line width=.6pt] (.5,2) node[draw,shape=circle,scale=.5] (T) {};
\draw[line width=.6pt] (-1.5,2) node[draw,shape=circle,scale=.5] (TR) {};
\draw[line width=.6pt] (-.5,4.5) node[draw,shape=circle,scale=.5] (TT) {};
\draw (.5,-.5) node[font=\scriptsize] {...};
\draw[line width=.6pt] (L) -- ++(0,.75) -- ++(-1,1) -- (T);
\draw[line width=.6pt] (R) -- ++(0,.75) -- ++(1,1);
\draw[line width=.6pt] (TR) -- ++(0,.75) -- ++(1,1) -- (TT);
\draw[line width=.6pt] (T) -- ++(0,.75) -- ++(-1,1);
\draw[line width=.6pt] (RR) -- (TR);
\end{tikzpicture}
\\
& \phantom{={}}
\;\oplus\;
\begin{tikzpicture}[baseline=15,x=.75em,y=.75em]
\draw[line width=.6pt] (-1.5,-.5) node[draw,shape=rectangle,scale=.6] (L) {};
\draw[line width=.6pt] (.5,-.5)  node[draw,shape=rectangle,scale=.6] (R) {};
\draw[line width=.6pt] (1.5,-.5)  node[draw,shape=diamond,scale=.43] (RR) {};
\draw[line width=.6pt] (-.5,2) node[draw,shape=circle,scale=.5] (T) {};
\draw[line width=.6pt] (1.5,2) node[draw,shape=circle,scale=.5] (TR) {};
\draw[line width=.6pt] (.5,4.5) node[draw,shape=circle,scale=.5] (TT) {};
\draw (-.5,-.5) node[font=\scriptsize] {...};
\draw[line width=.6pt] (L) -- ++(0,.75) -- ++(1,1) -- (T);
\draw[line width=.6pt] (R) -- ++(0,.75) -- ++(-1,1);
\draw[line width=.6pt] (TR) -- ++(0,.75) -- ++(-1,1) -- (TT);
\draw[line width=.6pt] (T) -- ++(0,.75) -- ++(1,1);
\draw[line width=.6pt] (RR) -- (TR);
\end{tikzpicture}
\;\oplus (-1)^m\;
\begin{tikzpicture}[baseline=15,x=.75em,y=.75em]
\draw[line width=.6pt] (1.5,-.5) node[draw,shape=rectangle,scale=.6] (L) {};
\draw[line width=.6pt] (-.5,-.5)  node[draw,shape=rectangle,scale=.6] (R) {};
\draw[line width=.6pt] (-1.5,-.5)  node[draw,shape=diamond,scale=.43] (RR) {};
\draw[line width=.6pt] (.5,2) node[draw,shape=circle,scale=.5] (T) {};
\draw[line width=.6pt] (-1.5,2) node[draw,shape=circle,scale=.5] (TR) {};
\draw[line width=.6pt] (-.5,4.5) node[draw,shape=circle,scale=.5] (TT) {};
\draw (.5,-.5) node[font=\scriptsize] {...};
\draw[line width=.6pt] (L) -- ++(0,.75) -- ++(-1,1) -- (T);
\draw[line width=.6pt] (R) -- ++(0,.75) -- ++(1,1);
\draw[line width=.6pt] (TR) -- ++(0,.75) -- ++(1,1) -- (TT);
\draw[line width=.6pt] (T) -- ++(0,.75) -- ++(-1,1);
\draw[line width=.6pt] (RR) -- (TR);
\end{tikzpicture}
\;\oplus\;
\begin{tikzpicture}[baseline=15,x=.75em,y=.75em]
\draw[line width=.6pt] (-1.5,-.5) node[draw,shape=diamond,scale=.43] (L) {};
\draw[line width=.6pt] (.5,-.5)  node[draw,shape=diamond,scale=.43] (R) {};
\draw[line width=.6pt] (1.5,-.5)  node[draw,shape=rectangle,scale=.6] (RR) {};
\draw[line width=.6pt] (-.5,2) node[draw,shape=circle,scale=.5] (T) {};
\draw[line width=.6pt] (1.5,2) node[draw,shape=circle,scale=.5] (TR) {};
\draw[line width=.6pt] (.5,4.5) node[draw,shape=circle,scale=.5] (TT) {};
\draw (-.5,-.5) node[font=\scriptsize] {...};
\draw[line width=.6pt] (L) -- ++(0,.75) -- ++(1,1) -- (T);
\draw[line width=.6pt] (R) -- ++(0,.75) -- ++(-1,1);
\draw[line width=.6pt] (TR) -- ++(0,.75) -- ++(-1,1) -- (TT);
\draw[line width=.6pt] (T) -- ++(0,.75) -- ++(1,1);
\draw[line width=.6pt] (RR) -- (TR);
\end{tikzpicture}
\;\oplus (-1)^m \;
\begin{tikzpicture}[baseline=15,x=.75em,y=.75em]
\draw[line width=.6pt] (1.5,-.5) node[draw,shape=diamond,scale=.43] (L) {};
\draw[line width=.6pt] (-.5,-.5)  node[draw,shape=diamond,scale=.43] (R) {};
\draw[line width=.6pt] (-1.5,-.5)  node[draw,shape=rectangle,scale=.6] (RR) {};
\draw[line width=.6pt] (.5,2) node[draw,shape=circle,scale=.5] (T) {};
\draw[line width=.6pt] (-1.5,2) node[draw,shape=circle,scale=.5] (TR) {};
\draw[line width=.6pt] (-.5,4.5) node[draw,shape=circle,scale=.5] (TT) {};
\draw (.5,-.5) node[font=\scriptsize] {...};
\draw[line width=.6pt] (L) -- ++(0,.75) -- ++(-1,1) -- (T);
\draw[line width=.6pt] (R) -- ++(0,.75) -- ++(1,1);
\draw[line width=.6pt] (TR) -- ++(0,.75) -- ++(1,1) -- (TT);
\draw[line width=.6pt] (T) -- ++(0,.75) -- ++(-1,1);
\draw[line width=.6pt] (RR) -- (TR);
\end{tikzpicture}
\end{align*}
The first line of $[M, a]$ and the second line of $[[M, a], \Phi]$ have mixed inputs and thus map to $0$ under $P$. Moreover, higher commutators with $\Phi$ vanish as $[[M, a], \Phi]$ has neither $A_U$ or $A_V$ outputs nor $A_W$ inputs ({\it i.e.}\ no \hair$\begin{tikzpicture}[baseline=-3,x=.75em,y=.75em] \draw[line width=.6pt] node[draw,shape=rectangle,scale=.6] {}; \end{tikzpicture}$\hspace{.05em} or
\hair$\begin{tikzpicture}[baseline=-3,x=.75em,y=.75em] \draw[line width=.6pt] node[draw,shape=diamond,scale=.43] {}; \end{tikzpicture}$\hspace{.05em} on the top, nor a
\hair$\begin{tikzpicture}[baseline=-3,x=.75em,y=.75em] \draw[line width=.6pt] node[draw,shape=circle,scale=.5] {}; \end{tikzpicture}$\hspace{.05em} on the bottom).

Thus
\begin{flalign*}
&& P_\Phi ([M, a]) &= 
(-1)^m \Bigg( \; \begin{tikzpicture}[baseline=15,x=.75em,y=.75em]
\draw[line width=.6pt] (1.5,-.5) node[draw,shape=rectangle,scale=.6] (L) {};
\draw[line width=.6pt] (-.5,-.5)  node[draw,shape=rectangle,scale=.6] (R) {};
\draw[line width=.6pt] (-1.5,-.5)  node[draw,shape=rectangle,scale=.6] (RR) {};
\draw[line width=.6pt] (.5,2) node[draw,shape=circle,scale=.5] (T) {};
\draw[line width=.6pt] (-1.5,2) node[draw,shape=circle,scale=.5] (TR) {};
\draw[line width=.6pt] (-.5,4.5) node[draw,shape=circle,scale=.5] (TT) {};
\draw (.5,-.5) node[font=\scriptsize] {...};
\draw[line width=.6pt] (L) -- ++(0,.75) -- ++(-1,1) -- (T);
\draw[line width=.6pt] (R) -- ++(0,.75) -- ++(1,1);
\draw[line width=.6pt] (TR) -- ++(0,.75) -- ++(1,1) -- (TT);
\draw[line width=.6pt] (T) -- ++(0,.75) -- ++(-1,1);
\draw[line width=.6pt] (RR) -- (TR);
\end{tikzpicture}
\;+
\sum_{i = 0}^m
(-1)^{i+1} \;
\begin{tikzpicture}[baseline=-3,x=.75em,y=.75em]
\draw[line width=.6pt] (-2,0) node[draw,shape=rectangle,scale=.6] (L) {};
\draw[line width=.6pt] (2,0)  node[draw,shape=rectangle,scale=.6] (R) {};
\draw[line width=.6pt] (0,0)  node[draw,shape=rectangle,scale=.6] (M) {};
\draw[line width=.6pt] (0,2.5) node[draw,shape=circle,scale=.5] (T) {};
\draw[line width=.6pt] (-.5,-2) node[draw,shape=rectangle,scale=.6] (BL) {};
\draw[line width=.6pt] (.5,-2) node[draw,shape=rectangle,scale=.6] (BR) {};
\draw (-1,0) node[font=\scriptsize] {...};
\draw (1,0) node[font=\scriptsize] {...};
\draw[line width=.6pt] (L) -- ++(0,.75) -- ++(2,1) -- (T);
\draw[line width=.6pt] (R) -- ++(0,.75) -- ++(-2,1);
\draw[line width=.6pt] (BL) -- ++(0,.75) -- ++(.5,.5) -- (M);
\draw[line width=.6pt] (BR) -- ++(0,.75) -- ++(-.5,.5);
\draw[line width=.6pt] (M) -- ++(0,.75) -- ++(0,1);
\end{tikzpicture}
\;+ (-1)^{m+2} \;
\begin{tikzpicture}[baseline=15,x=.75em,y=.75em]
\draw[line width=.6pt] (-1.5,-.5) node[draw,shape=rectangle,scale=.6] (L) {};
\draw[line width=.6pt] (.5,-.5)  node[draw,shape=rectangle,scale=.6] (R) {};
\draw[line width=.6pt] (1.5,-.5)  node[draw,shape=rectangle,scale=.6] (RR) {};
\draw[line width=.6pt] (-.5,2) node[draw,shape=circle,scale=.5] (T) {};
\draw[line width=.6pt] (1.5,2) node[draw,shape=circle,scale=.5] (TR) {};
\draw[line width=.6pt] (.5,4.5) node[draw,shape=circle,scale=.5] (TT) {};
\draw (-.5,-.5) node[font=\scriptsize] {...};
\draw[line width=.6pt] (L) -- ++(0,.75) -- ++(1,1) -- (T);
\draw[line width=.6pt] (R) -- ++(0,.75) -- ++(-1,1);
\draw[line width=.6pt] (TR) -- ++(0,.75) -- ++(-1,1) -- (TT);
\draw[line width=.6pt] (T) -- ++(0,.75) -- ++(1,1);
\draw[line width=.6pt] (RR) -- (TR);
\end{tikzpicture}
\;
\Bigg)
\\
&& & \phantom{={}}
\;\oplus
(-1)^m
\Bigg(
\;
\begin{tikzpicture}[baseline=15,x=.75em,y=.75em]
\draw[line width=.6pt] (1.5,-.5) node[draw,shape=diamond,scale=.43] (L) {};
\draw[line width=.6pt] (-.5,-.5)  node[draw,shape=diamond,scale=.43] (R) {};
\draw[line width=.6pt] (-1.5,-.5)  node[draw,shape=diamond,scale=.43] (RR) {};
\draw[line width=.6pt] (.5,2) node[draw,shape=circle,scale=.5] (T) {};
\draw[line width=.6pt] (-1.5,2) node[draw,shape=circle,scale=.5] (TR) {};
\draw[line width=.6pt] (-.5,4.5) node[draw,shape=circle,scale=.5] (TT) {};
\draw (.5,-.5) node[font=\scriptsize] {...};
\draw[line width=.6pt] (L) -- ++(0,.75) -- ++(-1,1) -- (T);
\draw[line width=.6pt] (R) -- ++(0,.75) -- ++(1,1);
\draw[line width=.6pt] (TR) -- ++(0,.75) -- ++(1,1) -- (TT);
\draw[line width=.6pt] (T) -- ++(0,.75) -- ++(-1,1);
\draw[line width=.6pt] (RR) -- (TR);
\end{tikzpicture}
\;+
\sum_{i = 0}^m
(-1)^{i+1}\;
\begin{tikzpicture}[baseline=-3,x=.75em,y=.75em]
\draw[line width=.6pt] (-2,0) node[draw,shape=diamond,scale=.43] (L) {};
\draw[line width=.6pt] (2,0)  node[draw,shape=diamond,scale=.43] (R) {};
\draw[line width=.6pt] (0,0)  node[draw,shape=diamond,scale=.43] (M) {};
\draw[line width=.6pt] (0,2.5) node[draw,shape=circle,scale=.5] (T) {};
\draw[line width=.6pt] (-.5,-2) node[draw,shape=diamond,scale=.43] (BL) {};
\draw[line width=.6pt] (.5,-2) node[draw,shape=diamond,scale=.43] (BR) {};
\draw (-1,0) node[font=\scriptsize] {...};
\draw (1,0) node[font=\scriptsize] {...};
\draw[line width=.6pt] (L) -- ++(0,.75) -- ++(2,1) -- (T);
\draw[line width=.6pt] (R) -- ++(0,.75) -- ++(-2,1);
\draw[line width=.6pt] (BL) -- ++(0,.75) -- ++(.5,.5) -- (M);
\draw[line width=.6pt] (BR) -- ++(0,.75) -- ++(-.5,.5);
\draw[line width=.6pt] (M) -- ++(0,.75) -- ++(0,1);
\end{tikzpicture}
\;+ (-1)^{m+2}\;
\begin{tikzpicture}[baseline=15,x=.75em,y=.75em]
\draw[line width=.6pt] (-1.5,-.5) node[draw,shape=diamond,scale=.43] (L) {};
\draw[line width=.6pt] (.5,-.5)  node[draw,shape=diamond,scale=.43] (R) {};
\draw[line width=.6pt] (1.5,-.5)  node[draw,shape=diamond,scale=.43] (RR) {};
\draw[line width=.6pt] (-.5,2) node[draw,shape=circle,scale=.5] (T) {};
\draw[line width=.6pt] (1.5,2) node[draw,shape=circle,scale=.5] (TR) {};
\draw[line width=.6pt] (.5,4.5) node[draw,shape=circle,scale=.5] (TT) {};
\draw (-.5,-.5) node[font=\scriptsize] {...};
\draw[line width=.6pt] (L) -- ++(0,.75) -- ++(1,1) -- (T);
\draw[line width=.6pt] (R) -- ++(0,.75) -- ++(-1,1);
\draw[line width=.6pt] (TR) -- ++(0,.75) -- ++(-1,1) -- (TT);
\draw[line width=.6pt] (T) -- ++(0,.75) -- ++(1,1);
\draw[line width=.6pt] (RR) -- (TR);
\end{tikzpicture}
\;
\Bigg)
\\
&& &= (-1)^m d_{\mathrm H} (a). && \qedhere
\end{flalign*}
\end{proof}

\begin{definition}
In light of Proposition \ref{unary}, we call the L$_\infty$ algebra $(\widetilde{\mathfrak g}[1] \oplus \mathfrak a)^{P_\Phi}_M$ given in Theorem \ref{maintheorem} parametrizing deformations of a diagram $\mathbb A$ of associative algebras on the category $U \tikzleftarrow W \tikzrightarrow V$ the {\it Gerstenhaber--Schack algebra} and denote it by $\mathfrak{gs} (\mathbb A)$.
\end{definition}

\begin{proposition}
\label{explicitmultibrackets}
The higher multibrackets of the Gerstenhaber--Schack algebra $\mathfrak{gs} (\mathbb A) = (\widetilde{\mathfrak g}[1] \oplus \mathfrak a)^{P_\Phi}_M$ may be given explicitly by the following formulae
\begin{alignat*}{5}
\lmb x[1], y[1] \rmb &= (-1)^{\lvert x \rvert + 1} [x, y][1] && \in \widetilde{\mathfrak g}[1] \\
\lmb x[1], a \rmb &= [x, a]^\Phi && \in \mathfrak a \\
\lmb x[1], a_1, a_2 \rmb &= \eqmakebox[xM0]{$x$} \circ (a_1 \otimes a_2) + \eqmakebox[xM0]{$x$} \circ (a_2 \otimes a_1) && \in \mathfrak a \\
\lmb a_1, a_2 \rmb &= \eqmakebox[xM0]{$M$} \circ (a_1 \otimes a_2) + \eqmakebox[xM0]{$M$} \circ (a_2 \otimes a_1) && \in \mathfrak a
\end{alignat*}
for homogeneous elements $x \in \widetilde{\mathfrak g}$ and $a, a_1, a_2 \in \mathfrak a$. A formula for $\lmb x[1], a_1, \dotsc, a_n \rmb \in \mathfrak a$ is given in the proof. Moreover,
\begin{align*}
\lmb x[1], a_1, \dotsc, a_n \rmb &= 0 \quad \text{if\, $n > \eqmakebox[xM3]{$\lvert x \rvert$} + 1$} \\
\lmb a_1, \dotsc, a_n \rmb &= 0 \quad \text{if\, $n > \eqmakebox[xM3]{$\lvert M \rvert$} + 1 = 2$.}
\end{align*}
\end{proposition}

\begin{proof}
Let $\Phi, x, a$ be as in (\ref{MPhixa}). The higher brackets are calculated similarly, giving
\begin{align*}
P_\Phi ([x, a]) &=
\sum_{i=0}^n (-1)^{mi}\;
\begin{tikzpicture}[baseline=-3,x=.75em,y=.75em]
\draw[line width=.6pt] (-2,0) node[draw,shape=circle,scale=.5] (L) {};
\draw[line width=.6pt] (2,0)  node[draw,shape=circle,scale=.5] (R) {};
\draw[line width=.6pt] (0,0)  node[draw,shape=circle,scale=.5] (M) {};
\draw[line width=.6pt] (0,2.5) node[draw,shape=circle,scale=.5] (T) {};
\draw[line width=.6pt] (-1,-2.5) node[draw,shape=rectangle,scale=.6] (BL) {};
\draw[line width=.6pt] (1,-2.5) node[draw,shape=rectangle,scale=.6] (BR) {};
\draw[line width=.6pt] (-2,-2.5) node[draw,shape=rectangle,scale=.6] (BLL) {};
\draw[line width=.6pt] (2,-2.5) node[draw,shape=rectangle,scale=.6] (BRR) {};
\draw (-1,0) node[font=\scriptsize] {...};
\draw (1,0) node[font=\scriptsize] {...};
\draw (0,-2.5) node[font=\scriptsize] {...};
\draw[line width=.6pt] (L) -- ++(0,.75) -- ++(2,1) -- (T);
\draw[line width=.6pt] (R) -- ++(0,.75) -- ++(-2,1);
\draw[line width=.6pt] (BL) -- ++(0,.75) -- ++(1,1) -- (M);
\draw[line width=.6pt] (BR) -- ++(0,.75) -- ++(-1,1);
\draw[line width=.6pt] (M) -- ++(0,.75) -- ++(0,1);
\draw[line width=.6pt] (BLL) -- (L);
\draw[line width=.6pt] (BRR) -- (R);
\end{tikzpicture}
\;
- (-1)^{mn} \sum_{i=0}^m (-1)^{ni}\;
\begin{tikzpicture}[baseline=-3,x=.75em,y=.75em]
\draw[line width=.6pt] (-2,0) node[draw,shape=rectangle,scale=.6] (L) {};
\draw[line width=.6pt] (2,0)  node[draw,shape=rectangle,scale=.6] (R) {};
\draw[line width=.6pt] (0,0)  node[draw,shape=rectangle,scale=.6] (M) {};
\draw[line width=.6pt] (0,2.5) node[draw,shape=circle,scale=.5] (T) {};
\draw[line width=.6pt] (-1,-2.5) node[draw,shape=rectangle,scale=.6] (BL) {};
\draw[line width=.6pt] (1,-2.5) node[draw,shape=rectangle,scale=.6] (BR) {};
\draw (-1,0) node[font=\scriptsize] {...};
\draw (1,0) node[font=\scriptsize] {...};
\draw (0,-2.5) node[font=\scriptsize] {...};
\draw[line width=.6pt] (L) -- ++(0,.75) -- ++(2,1) -- (T);
\draw[line width=.6pt] (R) -- ++(0,.75) -- ++(-2,1);
\draw[line width=.6pt] (BL) -- ++(0,.75) -- ++(1,1) -- (M);
\draw[line width=.6pt] (BR) -- ++(0,.75) -- ++(-1,1);
\draw[line width=.6pt] (M) -- ++(0,.75) -- ++(0,1);
\end{tikzpicture}
\\
&\phantom{={}}
\;\oplus
\sum_{i=0}^n (-1)^{mi}\;
\begin{tikzpicture}[baseline=-3,x=.75em,y=.75em]
\draw[line width=.6pt] (-2,0) node[draw,shape=circle,scale=.5] (L) {};
\draw[line width=.6pt] (2,0)  node[draw,shape=circle,scale=.5] (R) {};
\draw[line width=.6pt] (0,0)  node[draw,shape=circle,scale=.5] (M) {};
\draw[line width=.6pt] (0,2.5) node[draw,shape=circle,scale=.5] (T) {};
\draw[line width=.6pt] (-1,-2.5) node[draw,shape=diamond,scale=.43] (BL) {};
\draw[line width=.6pt] (1,-2.5) node[draw,shape=diamond,scale=.43] (BR) {};
\draw[line width=.6pt] (-2,-2.5) node[draw,shape=diamond,scale=.43] (BLL) {};
\draw[line width=.6pt] (2,-2.5) node[draw,shape=diamond,scale=.43] (BRR) {};
\draw (-1,0) node[font=\scriptsize] {...};
\draw (1,0) node[font=\scriptsize] {...};
\draw (0,-2.5) node[font=\scriptsize] {...};
\draw[line width=.6pt] (L) -- ++(0,.75) -- ++(2,1) -- (T);
\draw[line width=.6pt] (R) -- ++(0,.75) -- ++(-2,1);
\draw[line width=.6pt] (BL) -- ++(0,.75) -- ++(1,1) -- (M);
\draw[line width=.6pt] (BR) -- ++(0,.75) -- ++(-1,1);
\draw[line width=.6pt] (M) -- ++(0,.75) -- ++(0,1);
\draw[line width=.6pt] (BLL) -- (L);
\draw[line width=.6pt] (BRR) -- (R);
\end{tikzpicture}
\;
- (-1)^{mn} \sum_{i=0}^m (-1)^{ni}\;
\begin{tikzpicture}[baseline=-3,x=.75em,y=.75em]
\draw[line width=.6pt] (-2,0) node[draw,shape=diamond,scale=.43] (L) {};
\draw[line width=.6pt] (2,0)  node[draw,shape=diamond,scale=.43] (R) {};
\draw[line width=.6pt] (0,0)  node[draw,shape=diamond,scale=.43] (M) {};
\draw[line width=.6pt] (0,2.5) node[draw,shape=circle,scale=.5] (T) {};
\draw[line width=.6pt] (-1,-2.5) node[draw,shape=diamond,scale=.43] (BL) {};
\draw[line width=.6pt] (1,-2.5) node[draw,shape=diamond,scale=.43] (BR) {};
\draw (-1,0) node[font=\scriptsize] {...};
\draw (1,0) node[font=\scriptsize] {...};
\draw (0,-2.5) node[font=\scriptsize] {...};
\draw[line width=.6pt] (L) -- ++(0,.75) -- ++(2,1) -- (T);
\draw[line width=.6pt] (R) -- ++(0,.75) -- ++(-2,1);
\draw[line width=.6pt] (BL) -- ++(0,.75) -- ++(1,1) -- (M);
\draw[line width=.6pt] (BR) -- ++(0,.75) -- ++(-1,1);
\draw[line width=.6pt] (M) -- ++(0,.75) -- ++(0,1);
\end{tikzpicture}
\\
&= x \circ^\Phi a - (-1)^{mn} a \circ x.
\end{align*}
Note that this is essentially the Gerstenhaber bracket, replacing the identity by $\Phi$ wherever necessary, see Definition \ref{extendedgerstenhaber}.

Finally, $\lmb x[1], a_1, \dotsc, a_n \rmb = P_\Phi [...[[x, a_1], a_2], \dotsc, a_n]$ is obtained by plugging the outputs of $a_1, \dotsc, a_n$ into $n$ different inputs of $x$.
\begin{align*}
P_\Phi ([...[[x, a_1], a_2], \dotsc, a_n]) &=
\sum_I
\sum_{s \in \mathfrak S_I}
\epsilon (s) \,
\begin{tikzpicture}[baseline=20,x=.75em,y=.75em]
\draw[line width=.6pt] (-8,0) node[draw,shape=rectangle,scale=.6] (B1) {};
\draw[line width=.6pt] (-6,0) node[draw,shape=rectangle,scale=.6] (B2) {};
\draw[line width=.6pt] (-4,0) node[draw,shape=rectangle,scale=.6] (B3) {};
\draw[line width=.6pt] (-2,0) node[draw,shape=rectangle,scale=.6] (B4) {};
\draw[line width=.6pt] (0,0)  node[draw,shape=rectangle,scale=.6] (B5) {};
\draw[line width=.6pt] (2,0)  node[draw,shape=rectangle,scale=.6] (B6) {};
\draw[line width=.6pt] (4,0)  node[draw,shape=rectangle,scale=.6] (B7) {};
\draw[line width=.6pt] (6,0)  node[draw,shape=rectangle,scale=.6] (B8) {};
\draw[line width=.6pt] (8,0)  node[draw,shape=rectangle,scale=.6] (B9) {};
\draw[line width=.6pt] (-8,2.5) node[draw,shape=circle,scale=.5] (M1) {};
\draw[line width=.6pt] (-5,2.5) node[draw,shape=circle,scale=.5] (M2) {};
\draw[line width=.6pt] (-2,2.5) node[draw,shape=circle,scale=.5] (M3) {};
\draw[line width=.6pt] (1,2.5)  node[draw,shape=circle,scale=.5] (M4) {};
\draw[line width=.6pt] (5,2.5)  node[draw,shape=circle,scale=.5] (M5) {};
\draw[line width=.6pt] (8,2.5)  node[draw,shape=circle,scale=.5] (M6) {};
\draw[line width=.6pt] (0,6)  node[draw,shape=circle,scale=.5] (T) {};
\draw (-7,0) node[font=\scriptsize] {...};
\draw (-5,0) node[font=\scriptsize] {...};
\draw (-3,0) node[font=\scriptsize] {...};
\draw (-1,0) node[font=\scriptsize] {...};
\draw (1,0)  node[font=\scriptsize] {...};
\draw (3,0)  node[font=\scriptsize] {...};
\draw (5,0)  node[font=\scriptsize] {...};
\draw (7,0)  node[font=\scriptsize] {...};
\draw (-6.5,2.5) node[font=\scriptsize] {........};
\draw (-3.5,2.5) node[font=\scriptsize] {........};
\draw (-0.5,2.5) node[font=\scriptsize] {........};
\draw (3,2.5) node[font=\scriptsize] {............};
\draw (6.5,2.5) node[font=\scriptsize] {........};
\draw (-4.9,-0.75) node[font=\scriptsize] {$a_{s(1)}$};
\draw (1.1,-0.75)  node[font=\scriptsize] {$a_{s(2)}$};
\draw (5.1,-0.75)  node[font=\scriptsize] {$a_{s(n)}$};
\draw[line width=.6pt] (B1) -- (M1);
\draw[line width=.6pt] (B2) -- ++(0,.75) -- ++(1,1) -- (M2);
\draw[line width=.6pt] (B3) -- ++(0,.75) -- ++(-1,1);
\draw[line width=.6pt] (B4) -- (M3);
\draw[line width=.6pt] (B5) -- ++(0,.75) -- ++(1,1) -- (M4);
\draw[line width=.6pt] (B6) -- ++(0,.75) -- ++(-1,1);
\draw[line width=.6pt] (B7) -- ++(0,.75) -- ++(1,1) -- (M5);
\draw[line width=.6pt] (B8) -- ++(0,.75) -- ++(-1,1);
\draw[line width=.6pt] (B9) -- (M6);
\draw[line width=.6pt] (M1) -- ++(0,.75) -- ++(8,2) -- (T);
\draw[line width=.6pt] (M2) -- ++(0,.75) -- ++(5,2);
\draw[line width=.6pt] (M3) -- ++(0,.75) -- ++(2,2);
\draw[line width=.6pt] (M4) -- ++(0,.75) -- ++(-1,2);
\draw[line width=.6pt] (M5) -- ++(0,.75) -- ++(-5,2);
\draw[line width=.6pt] (M6) -- ++(0,.75) -- ++(-8,2);
\end{tikzpicture}
\\
& \phantom{={}}
\;\oplus
\sum_I
\sum_{s \in \mathfrak S_I}
\epsilon (s) \,
\begin{tikzpicture}[baseline=20,x=.75em,y=.75em]
\draw[line width=.6pt] (-8,0) node[draw,shape=diamond,scale=.43] (B1) {};
\draw[line width=.6pt] (-6,0) node[draw,shape=diamond,scale=.43] (B2) {};
\draw[line width=.6pt] (-4,0) node[draw,shape=diamond,scale=.43] (B3) {};
\draw[line width=.6pt] (-2,0) node[draw,shape=diamond,scale=.43] (B4) {};
\draw[line width=.6pt] (0,0)  node[draw,shape=diamond,scale=.43] (B5) {};
\draw[line width=.6pt] (2,0)  node[draw,shape=diamond,scale=.43] (B6) {};
\draw[line width=.6pt] (4,0)  node[draw,shape=diamond,scale=.43] (B7) {};
\draw[line width=.6pt] (6,0)  node[draw,shape=diamond,scale=.43] (B8) {};
\draw[line width=.6pt] (8,0)  node[draw,shape=diamond,scale=.43] (B9) {};
\draw[line width=.6pt] (-8,2.5) node[draw,shape=circle,scale=.5] (M1) {};
\draw[line width=.6pt] (-5,2.5) node[draw,shape=circle,scale=.5] (M2) {};
\draw[line width=.6pt] (-2,2.5) node[draw,shape=circle,scale=.5] (M3) {};
\draw[line width=.6pt] (1,2.5)  node[draw,shape=circle,scale=.5] (M4) {};
\draw[line width=.6pt] (5,2.5)  node[draw,shape=circle,scale=.5] (M5) {};
\draw[line width=.6pt] (8,2.5)  node[draw,shape=circle,scale=.5] (M6) {};
\draw[line width=.6pt] (0,6)  node[draw,shape=circle,scale=.5] (T) {};
\draw (-7,0) node[font=\scriptsize] {...};
\draw (-5,0) node[font=\scriptsize] {...};
\draw (-3,0) node[font=\scriptsize] {...};
\draw (-1,0) node[font=\scriptsize] {...};
\draw (1,0)  node[font=\scriptsize] {...};
\draw (3,0)  node[font=\scriptsize] {...};
\draw (5,0)  node[font=\scriptsize] {...};
\draw (7,0)  node[font=\scriptsize] {...};
\draw (-6.5,2.5) node[font=\scriptsize] {........};
\draw (-3.5,2.5) node[font=\scriptsize] {........};
\draw (-0.5,2.5) node[font=\scriptsize] {........};
\draw (3,2.5) node[font=\scriptsize] {............};
\draw (6.5,2.5) node[font=\scriptsize] {........};
\draw (-4.9,-0.75) node[font=\scriptsize] {$a_{s(1)}$};
\draw (1.1,-0.75)  node[font=\scriptsize] {$a_{s(2)}$};
\draw (5.1,-0.75)  node[font=\scriptsize] {$a_{s(n)}$};
\draw[line width=.6pt] (B1) -- (M1);
\draw[line width=.6pt] (B2) -- ++(0,.75) -- ++(1,1) -- (M2);
\draw[line width=.6pt] (B3) -- ++(0,.75) -- ++(-1,1);
\draw[line width=.6pt] (B4) -- (M3);
\draw[line width=.6pt] (B5) -- ++(0,.75) -- ++(1,1) -- (M4);
\draw[line width=.6pt] (B6) -- ++(0,.75) -- ++(-1,1);
\draw[line width=.6pt] (B7) -- ++(0,.75) -- ++(1,1) -- (M5);
\draw[line width=.6pt] (B8) -- ++(0,.75) -- ++(-1,1);
\draw[line width=.6pt] (B9) -- (M6);
\draw[line width=.6pt] (M1) -- ++(0,.75) -- ++(8,2) -- (T);
\draw[line width=.6pt] (M2) -- ++(0,.75) -- ++(5,2);
\draw[line width=.6pt] (M3) -- ++(0,.75) -- ++(2,2);
\draw[line width=.6pt] (M4) -- ++(0,.75) -- ++(-1,2);
\draw[line width=.6pt] (M5) -- ++(0,.75) -- ++(-5,2);
\draw[line width=.6pt] (M6) -- ++(0,.75) -- ++(-8,2);
\end{tikzpicture}
\end{align*}
where we sum over subsets $I \subset \{ 1, \dotsc, m \}$ of length $n$ and $\mathfrak S_I$ denotes the permutation group of $I$. (Here, $\epsilon (s)$ is the Koszul sign as in Definition \ref{linfinityalgebra}.)

For $x \in \widetilde{\mathfrak g}^1$ and $a_1, a_2 \in \mathfrak a^0$
\begin{flalign*}
&& \lmb x, a_1, a_2 \rmb &= P_\Phi ([[\eqmakebox[xM1]{$x,{}$} {} a_1], a_2]) = \eqmakebox[xM2]{$x$} \circ (a_1 \otimes a_2) + \eqmakebox[xM2]{$x$} \circ (a_2 \otimes a_1) \notag\\
&&    \lmb a_1, a_2 \rmb &= P_\Phi ([[\eqmakebox[xM1]{$M,{}$} {} a_1], a_2]) = \eqmakebox[xM2]{$M$} \circ (a_1 \otimes a_2) + \eqmakebox[xM2]{$M$} \circ (a_2 \otimes a_1). && \qedhere
\end{flalign*}
\end{proof}

In \S \ref{deformationsofcoh} we consider applications to algebraic geometry by considering a diagram of the form $\mathbb A = \mathcal O_X \vert_{\mathfrak U}$, describing deformations of the Abelian category $\Coh (X)$.

\subsection{Obstruction theory via L$_\infty$ algebras}
\label{obstructiontheorylinfinity}

Let $\mathbb A$ be a diagram of associative algebras on $\mathfrak U = \big( U \leftarrow W \tikzrightarrow V \big)$ and $\mathbb A \llrr{\epsilon}$ be defined by
\[
\mathbb A \llrr{\epsilon} (U) = \mathbb A (U) \hatotimes \mathbb k \llrr{\epsilon}.
\]
If $\mu, \nu, \xi$ denote the associative multiplications on $\mathbb A (U), \mathbb A (V), \mathbb A (W)$, respectively, let $\mu_0, \nu_0, \xi_0$ denote their obvious extensions to $\mathbb A \llrr{\epsilon} (U), \mathbb A \llrr{\epsilon} (V), \mathbb A \llrr{\epsilon} (W)$, respectively.

Similarly, let $\phi \colon \mathbb A (U) \tikzto \mathbb A (W)$ and $\psi \colon \mathbb A (V) \tikzto \mathbb A (W)$ be the images of the morphisms in $\mathfrak U$ under $\mathbb A$, which extend to morphisms $\phi_0 \colon \mathbb A \llrr{\epsilon} (U) \tikzto \mathbb A \llrr{\epsilon} (W)$ and $\psi_0 \colon \mathbb A \llrr{\epsilon} (V) \tikzto \mathbb A \llrr{\epsilon} (W)$.

Now set
\begin{align*}
M = \eqmakebox[MPhi0]{$M_0$} &= \mu_0 \oplus \nu_0 \oplus \xi_0 \\
\Phi = \eqmakebox[MPhi0]{$\Phi_0$} &= \phi_0 \oplus \psi_0
\end{align*}
so that $M[1] \oplus \Phi \in \mathfrak{gs} (\mathbb A \llrr{\epsilon})$ of degree $0$.

A formal deformation of $\mathbb A$ is given by $M + \widetilde M$ and $\Phi + \widetilde \Phi$, where $\widetilde M = \sum_{n \geq 1} M_n \epsilon^n$ and $\widetilde \Phi = \sum_{n \geq 1} \Phi_n \epsilon^n$ for
\begin{align*}
M_n &= \mu_n \oplus \nu_n \oplus \xi_n \\
\Phi_n &= \phi_n \oplus \psi_n
\end{align*}
with $\widetilde M[1] \oplus \widetilde \Phi$ of degree $0$.

By Proposition \ref{equivalent}, $M + \widetilde M$ and $\Phi + \widetilde \Phi$ are compatible precisely when $\widetilde M[1] \oplus \widetilde \Phi \in \MC (\mathfrak{gs} (\mathbb A) \hatotimes \mathfrak m)$ where $\mathfrak m = (\epsilon)$ is the maximal ideal of $\mathbb k \llrr{\epsilon}$, {\it i.e.}\ when
\[
\sum_{n \geq 0} \frac{\big( \widetilde M[1] \oplus \widetilde \Phi \big)^{\lmb n \rmb}}{n!} = 0
\]
which is equivalent to
\begin{align}
d_{\mathrm H} \widetilde M + \tfrac12 [\widetilde M, \widetilde M] &= 0 \, \in \widetilde{\mathfrak g}^1 \label{maurercartanzk1} \\
d_{\mathrm H} \widetilde \Phi + d_\Delta \widetilde M [1] + \lmb \widetilde M[1], \widetilde \Phi \rmb + \tfrac12 \lmb \widetilde \Phi, \widetilde \Phi \rmb + \tfrac12 \lmb \widetilde M[1], \widetilde \Phi, \widetilde \Phi \rmb &= 0 \, \in \mathfrak a^0. \label{maurercartanzk2}
\end{align}
Using the explicit formulae given in Proposition \ref{explicitmultibrackets}, we may rewrite the brackets in (\ref{maurercartanzk2}) as follows
\begin{align*}
d_\Delta (\widetilde M [1]) &= \widetilde M \circ (\Phi \otimes \Phi) - \Phi \circ \widetilde M \\
d_{\mathrm H} \widetilde \Phi &= M \circ (\Phi \otimes \widetilde \Phi) - \widetilde \Phi \circ M + M \circ (\widetilde \Phi \otimes \Phi) \\
\lmb \widetilde M [1], \widetilde \Phi \rmb &= \widetilde M \circ (\widetilde \Phi \otimes \Phi) + \widetilde M \circ (\Phi \otimes \widetilde \Phi) - \widetilde \Phi \circ \widetilde M \\
\tfrac12 \lmb \widetilde \Phi, \widetilde \Phi \rmb &= M \circ (\widetilde \Phi \otimes \widetilde \Phi) \\
\tfrac12 \lmb \widetilde M[1], \widetilde \Phi, \widetilde \Phi \rmb &= \widetilde M \circ (\widetilde \Phi \otimes \widetilde \Phi).
\end{align*}
Thus, we have that $\widetilde M[1] \oplus \widetilde \Phi \in \MC (\mathfrak{gs} (\mathbb A))$ precisely when ({\it i}\hair) $M + \widetilde M$ is a triple of associative multiplications which is equivalent to (\ref{maurercartanzk1}), and ({\it ii}\hair)
\begin{equation}
\label{obstructiontilde}
\begin{aligned}
\widetilde M \circ (\Phi \otimes \Phi) &{} + M \circ (\Phi \otimes \widetilde \Phi) + M \circ (\widetilde \Phi \otimes \Phi) \\
&{} + \widetilde M \circ (\widetilde \Phi \otimes \Phi) + \widetilde M \circ (\Phi \otimes \widetilde \Phi) \\
&{} + M \circ (\widetilde \Phi \otimes \widetilde \Phi) + \widetilde M \circ (\widetilde \Phi \otimes \widetilde \Phi) = \Phi \circ \widetilde M + \widetilde \Phi \circ M + \widetilde \Phi \circ \widetilde M
\end{aligned}
\end{equation}
which is equivalent to (\ref{maurercartanzk2}).

Moreover, collecting powers of $\epsilon$, we can rewrite (\ref{obstructiontilde}) as
\[
\sum_{i+j = n} \Phi_i \circ M_j = \sum_{i+j+k = n} M_i \circ (\Phi_j \otimes \Phi_k).
\]

\begin{remark}
Note that, unlike the Maurer--Cartan equation for a \textsc{dg} Lie algebra, the Maurer--Cartan equation for $\mathfrak{gs} (\mathbb A)$ contains a ternary bracket (\ref{maurercartanzk2}).
\end{remark}

\section{Applications to geometry}
\label{deformationsofcoh}

In this final section we wish to use the L$_\infty$ structure on the Gerstenhaber--Schack complex constructed in \S \ref{linfinitygerstenhaberschack} to study the deformation theory of the diagram $\mathcal O_X \vert_{\mathfrak U}$ for a (smooth) variety $X$, which by \cite{lowenvandenbergh1} is equivalent to the deformation theory of the Abelian category of (quasi)coherent sheaves $\Coh (X)$ or $\QCoh (X)$ (see also Lowen \cite[Thm.\ 1.4]{lowen}). As we made assumptions on the shape of the diagram, we first determine what varieties admit a cover by two acyclic opens. Since we have in mind applications to smooth complex surfaces, we state the applications for this case.

\paragraph{Curves.} 

Since a smooth projective curve $C$ minus finitely many points is affine --- hence acyclic by Serre's criterion \cite[Thm.\ III.3.7]{hartshorne} --- $C$ can be covered by two affine open sets (the complements of two distinct points on $C$).

Since also $\extprod^i \mathcal T_C = 0$ for $i > 1$ we have that deformations of $\mathcal O_C \vert_{\mathfrak U}$ are parametrized by
\begin{equation}
\label{hhcurves}
\H_{\mathrm{GS}}^2 (C) \simeq \HH^2 (C) \simeq \H^1 (C, \mathcal T_C)
\end{equation}
and all obstructions vanish. It is well known that
\begin{equation}
\label{hcurves}
\dim \H^1 (C, \mathcal T_C) =
\begin{cases}
0 & \text{genus $0$} \\
1 & \text{genus $1$} \\
3g-3 & \text{genus $g \geq 2$.}
\end{cases}
\end{equation}
and deformations of $\mathcal O_C \vert_{\mathfrak U}$ capture precisely deformations of the complex structure of $X$.

\paragraph{Surfaces and higher dimensions.}

First note that if a smooth scheme admits an open cover $\mathfrak U = \{ U, V \}$ of {\it two} affine opens with affine intersection, Leray's theorem (see for example \cite{griffithsharris}, Ch.~0, \S 3) states that the sheaf cohomology of $X$ can be calculated as the Čech cohomology of the open cover $\mathfrak U$, whence $\H^i (X, \mathcal F) \simeq \check \H^i (\mathfrak U, \mathcal F \vert_{\mathfrak U}) = 0$ for every $i \geq 2$ and every coherent sheaf $\mathcal F$.

Yet, if $X$ is smooth projective of $\dim_{\mathbb C} = n \geq 2$, Serre duality gives $\H^n (X, \omega) \simeq \H^0 (X, \mathcal O_X)^* \simeq \mathbb C \neq 0$. Thus $X$ cannot be covered by two acyclic open sets and applications in complex dimension $\geq 2$ are thus limited to noncompact spaces.

\begin{remark}
At least with respect to the original motivation of developing a tool for studying noncommutative instantons on complex surfaces this should not necessarily be seen as a drawback, as theories of instantons are often defined over {\it noncompact} spaces. Indeed, in the context of the instanton partition function defined by Nekrasov \cite{nekrasov} and explored by various authors (see for example \cite{nakajimayoshioka,bravermanetingof,nekrasovokounkov,gasparimliu}), the noncompactness of the underlying surface is essential for the nontriviality of the theory.
\end{remark}

A broad class of smooth complex varieties of dimension $\geq 2$ which can be covered by two affine open sets is given by the varieties $Z = \Tot \mathcal E$ for $\mathcal E$ an algebraic vector bundle over a smooth projective curve $C$ (of any genus).

For such $Z$, we thus get that deformations of $\Coh (Z)$ are parametrized by
\begin{align*}
\H^2_{\mathrm{GS}} (\mathcal O_Z \vert_{\mathfrak U}) &\simeq \HH^2 (Z) \simeq \H^0 (Z, \extprod^2 \mathcal T_Z) \oplus \H^1 (Z, \mathcal T_Z) \\
\intertext{with obstructions in}
\H^3_{\mathrm{GS}} (\mathcal O_Z \vert_{\mathfrak U}) &\simeq \HH^3 (Z) \simeq \H^0 (Z, \extprod^3 \mathcal T_Z) \oplus \H^1 (Z, \extprod^2 \mathcal T_Z).
\end{align*}
In particular, if $\mathcal E$ is a line bundle, then $Z = \Tot \mathcal E$ is a surface covered by two acyclic open sets. But then $\extprod^3 \mathcal T_X = 0$ and $\H^1 (X, \extprod^2 \mathcal T_X)$ is its only obstruction space.

\begin{remark}
For $Z = \Tot \mathcal E$, the cohomology groups may be infinite-dimensional over $\mathbb C$. This can be avoided by considering the $n$th formal neighbourhood of $C$ inside $Z$, {\it i.e.}\ by considering the reduced scheme with structure sheaf $\mathcal O_Z / \mathcal I_C^{n+1}$, where $\mathcal I_C$ is the ideal sheaf of $C$ in $Z$.
\end{remark}

In the introduction we mentioned that a deformation of $\Coh (X)$ or $\mathcal O_X \vert_{\mathfrak U}$ can be thought of as a simultaneous deformation quantization of $\mathcal O_X$ and deformation of the complex structure of $X$. We first explain how these can be reformulated in terms of deformations of diagrams. Throughout we will illustrate the theory by means of the noncompact surfaces $Z_k := \Tot \mathcal O_{\mathbb P^1} (-k)$ for $k \geq 1$, which admit both classical and noncommutative deformations for $k \geq 2$.

\subsection{Deformation quantization}
\label{deformationquantization}

Let $X$ be a smooth complex algebraic variety and let $\eta \in \H^0 (X, \extprod^2 \mathcal T_X)$ be a global bivector field with vanishing Schouten--Nijenhuis bracket $[\eta, \eta] \in \H^0 (X, \extprod^3 \mathcal T_X)$. Then $\eta$ defines a holomorphic Poisson structure $\{ \blank {,} \blank \}_\eta$ by
\[
\{ f, g \}_\eta = \langle \eta, \mathrm d f \wedge \mathrm d g \rangle,
\]
where $\mathrm d$ denotes the exterior derivative and $\langle \blank {,} \blank \rangle$ the pairing between vector fields and forms.

A {\it star product} on a complex variety $X$ is a $\mathbb C \llrr{\hbar}$-bilinear associative multiplication
\[
\star \colon \mathcal O_X \llrr{\hbar} \times \mathcal O_X \llrr{\hbar} \tikzto \mathcal O_X \llrr{\hbar}
\]
mapping
\[
(f, g) \tikzmapsto f \star g = fg + \hbar B_1 (f, g) + \hbar^2 B_2 (f, g) + \dotsb
\]
where $B_n$ are bidifferential operators. (Here $\hbar$ is a formal parameter, which we call $\epsilon$ elsewhere.)

A {\it deformation quantization} of a complex Poisson variety $(X, \eta)$ is a star product on $X$ with first term $B_1 (f, g) = \{ f, g \}_\eta$.

A deformation quantization of $X$ is thus given as a collection of bilinear maps $(B_n)_{n \geq 1} = (B_n^U)_{n \geq 1}$ for each open set $U$, deforming the commutative product of functions on $U$.

The condition that the star product on the individual open sets is compatible with the algebra maps $\phi_0 \colon \mathcal O \llrr{\hbar} (U) \tikzto \mathcal O_X \llrr{\hbar} (W)$ for every inclusion $W \subset U$ of open sets, can be written in terms of the bilinear operators $B_n$ for $n \geq 1$ as
\begin{equation}
\label{obstructionmultiplication}
B_n^{W} (\phi_0 (\blank), \phi_0 (\blank)) = \phi_0 (B_n^U (\blank {,} \blank)).
\end{equation}

The restriction of $(\mathcal O_X \llrr{\hbar}, \star)$ to an open cover $\mathfrak U$ of acyclic open sets gives a formal deformation of the diagram $\mathcal O_X \vert_{\mathfrak U}$ with parameter $\hbar$.

As part of the proof of his Formality Conjecture in \cite{kontsevich1}, Kontsevich gave an explicit construction of a star product on $\mathbb R^d$. This star product works equally on $\mathbb C^d$ and quantizations of more general algebraic varieties were studied in \cite{kontsevich2,vandenbergh,yekutieli}.

For a Poisson structure $\eta$ on $\mathbb C^d$, the {\it Kontsevich star product} $\star^{\mathrm K}$ is defined by
\begin{equation}
\label{kontsevichstar}
f \star^{\mathrm K} g = f g + \sum_{n \geq 1} \hbar^n \sum_{\Gamma \in \mathfrak G_{n,2}} w_\Gamma \, B_\Gamma (f, g)
\end{equation}
where $\Gamma$ is an ``admissible'' graph, $\mathfrak G_{n,2}$ is a finite set of such graphs, $B_\Gamma$ are bidifferential operators constructed from $\Gamma$ and $w_\Gamma$ is the {\it weight} of $\Gamma$ defined in terms of integral over a certain configuration space.

\begin{lemma}[\cite{dito}]
\label{secondorder}
Up to second order in $\hbar$ the Kontsevich star product for $\eta$ on $\mathbb C^d$ is given by
\begin{flalign*}
&& f \star^{\mathrm K} g ={}& f g \\
&& &+ \hbar \displaystyle\sum_{i,j} \eta^{ij} \; \partial_i(f) \; \partial_j(g) &
\mathllap{\begin{tikzpicture}[x=2em,y=2em,baseline=.7em]
\draw[line width=.5pt, fill=black] (1,0) circle(0.25ex);
\draw[line width=.5pt, fill=black] (0,0) circle(0.25ex);
\draw[line width=.5pt, fill=white] (.5,.886) circle(0.25ex);
\node[shape=circle, scale=0.8](f) at (0,0) {};
\node[shape=circle, scale=0.8](g) at (1,0) {};
\node[shape=circle, scale=0.8](one) at (.5,.886) {};
\path[-stealth, line width=.6pt, line cap=round] (one) edge (f);
\path[-stealth, line width=.6pt, line cap=round] (one) edge (g);
\end{tikzpicture}}
\\
&& &+ \frac{\hbar^2}2 \sum_{i,j,k,l} \eta^{ij} \eta^{kl} \; \partial_i \partial_k (f) \; \partial_j \partial_l (g) &
\mathllap{\begin{tikzpicture}[x=2em,y=2em,baseline=.7em]
\draw[line width=.5pt, fill=black] (1,0) circle(0.25ex);
\draw[line width=.5pt, fill=black] (0,0) circle(0.25ex);
\draw[line width=.5pt, fill=white] (0,1) circle(0.25ex);
\draw[line width=.5pt, fill=white] (1,1) circle(0.25ex);
\node[shape=circle, scale=0.7](f) at (0,0) {};
\node[shape=circle, scale=0.7](g) at (1,0) {};
\node[shape=circle, scale=0.7](L) at (0,1) {};
\node[shape=circle, scale=0.7](R) at (1,1) {};
\path[-stealth, line width=.5pt, line cap=round] (L) edge (f);
\path[-stealth, line width=.5pt, line cap=round] (L) edge (g);
\path[-stealth, line width=.5pt, line cap=round] (R) edge (f);
\path[-stealth, line width=.5pt, line cap=round] (R) edge (g);
\end{tikzpicture}}
\\
&& &+ \frac{\hbar^2}3 \sum_{i,j,k,l} \eta^{ij} \; \partial_i (\eta^{kl}) \; \partial_j \partial_l (f) \; \partial_k (g) &
\mathllap{\begin{tikzpicture}[x=2em,y=2em,baseline=.7em]
\draw[line width=.6pt, fill=black] (1,0) circle(0.25ex);
\draw[line width=.6pt, fill=black] (0,0) circle(0.25ex);
\draw[line width=.6pt, fill=white] (0,1) circle(0.25ex);
\draw[line width=.6pt, fill=white] (1,1) circle(0.25ex);
\node[shape=circle, scale=0.7](f) at (0,0) {};
\node[shape=circle, scale=0.7](g) at (1,0) {};
\node[shape=circle, scale=0.7](L) at (0,1) {};
\node[shape=circle, scale=0.7](R) at (1,1) {};
\path[-stealth, line width=.6pt, line cap=round] (L) edge (f);
\path[-stealth, line width=.6pt, line cap=round] (L) edge (R);
\path[-stealth, line width=.6pt, line cap=round] (R) edge (f);
\path[-stealth, line width=.6pt, line cap=round] (R) edge (g);
\end{tikzpicture}}
\\
&& &+ \frac{\hbar^2}3 \sum_{i,j,k,l} \eta^{kl} \; \partial_k (\eta^{ij}) \; \partial_i (f) \; \partial_j \partial_l (g) &
\mathllap{\begin{tikzpicture}[x=2em,y=2em,baseline=.7em]
\draw[line width=.6pt, fill=black] (1,0) circle(0.25ex);
\draw[line width=.6pt, fill=black] (0,0) circle(0.25ex);
\draw[line width=.6pt, fill=white] (0,1) circle(0.25ex);
\draw[line width=.6pt, fill=white] (1,1) circle(0.25ex);
\node[shape=circle, scale=0.7](f) at (0,0) {};
\node[shape=circle, scale=0.7](g) at (1,0) {};
\node[shape=circle, scale=0.7](L) at (0,1) {};
\node[shape=circle, scale=0.7](R) at (1,1) {};
\path[-stealth, line width=.6pt, line cap=round] (R) edge (g);
\path[-stealth, line width=.6pt, line cap=round] (R) edge (L);
\path[-stealth, line width=.6pt, line cap=round] (L) edge (g);
\path[-stealth, line width=.6pt, line cap=round] (L) edge (f);
\end{tikzpicture}}
\\
&& &- \frac{\hbar^2}6 \sum_{i,j,k,l} \partial_l (\eta^{ij}) \; \partial_j (\eta^{kl}) \; \partial_i (f) \; \partial_k (g) &
\mathllap{\begin{tikzpicture}[x=2em,y=2em,baseline=.7em]
\draw[line width=.6pt, fill=black] (1,0) circle(0.25ex);
\draw[line width=.6pt, fill=black] (0,0) circle(0.25ex);
\draw[line width=.6pt, fill=white] (0,1) circle(0.25ex);
\draw[line width=.6pt, fill=white] (1,1) circle(0.25ex);
\node[shape=circle, scale=0.7](f) at (0,0) {};
\node[shape=circle, scale=0.7](g) at (1,0) {};
\node[shape=circle, scale=0.7](L) at (0,1) {};
\node[shape=circle, scale=0.7](R) at (1,1) {};
\path[-stealth, line width=.6pt, line cap=round] (L) edge[out=15,   in=165] (R);
\path[-stealth, line width=.6pt, line cap=round] (L) edge (f);
\path[-stealth, line width=.6pt, line cap=round] (R) edge[out=-165, in=-15] (L);
\path[-stealth, line width=.6pt, line cap=round] (R) edge (g);
\end{tikzpicture}}
\\
&& &+ \dotsb
\end{flalign*}
\end{lemma}

To illustrate this, we consider the noncompact complex surfaces $Z_k := \Tot \mathcal O_{\mathbb P^1} (-k)$ for $k \geq 1$.

\begin{notation}
\label{canonicalcoordinates}
Cover $Z_k$ by two open sets $U = \{ (z, u) \in \mathbb C^2 \}$ and $V = \{ (\zeta, v) \in \mathbb C^2 \}$ such that on $U \cap V \simeq \mathbb C^* \times \mathbb C$ we identify
\[
(\zeta, v) \mapsto (z^{-1}, z^k u).
\]
We refer to these as {\it canonical coordinates}.
\end{notation}

Here our aim is to give one particular Poisson structure $\eta \in \H^0 (Z_k, \extprod^2 \mathcal T_{Z_k})$ for each $k \geq 1$ and explain how it can be quantized. For a detailed study of deformation quantizations of the surfaces $Z_k$ we refer to \cite{barmeiergasparim2}.

Since $Z_k$ is a surface, $\extprod^3 \mathcal T_{Z_k} = 0$ and thus any $\eta \in \H^0 (Z_k, \mathcal T_{Z_k})$ defines a global Poisson structure on $Z_k$.

We give Poisson structures on $Z_k$ in canonical coordinates (Notation \ref{canonicalcoordinates}). In particular, on $U = \{ (z, u) \in \mathbb C^2 \}$ a (holomorphic) Poisson structure can be given as a bivector field $f_U \, \partial_z \wedge \partial_u$, where $f_U$ is some holomorphic function on $U$ and similarly on $V = \{ (\zeta, v) \in \mathbb C^2 \}$. A global Poisson structure may thus be given by a pair $(f_U \, \partial_z \wedge \partial_u, f_V \, \partial_\zeta \wedge \partial_v)$ such that rewriting $f_V$ in terms of $z$ and $u$ via the change of coordinates $(\zeta, v) = (z^{-1}, z^k u)$ is equal to $-z^{k-2} f_U$. (Here $-z^{k-2}$ is the transition function of the anticanonical line bundle $\extprod^2 \mathcal T_{Z_k}$ written in canonical coordinates.) When referring to a Poisson structure in canonical coordinates, we shall often write only its coefficient functions as a pair $(f_U, f_V)$.

\begin{lemma}[\cite{barmeiergasparim2}]
\label{poissonzk}
$\H^0 (Z_k, \extprod^2 \mathcal T_{Z_k})$ is a finitely generated module over the algebra of global functions on $Z_k$. In canonical coordinates, generators can be given by
\begin{alignat*}{5}
(1) \quad & (1, -\zeta), (z, -1)                                && \text{for $k = 1$} \\
(2) \quad & (1, -1)                                             && \text{for $k = 2$} \\
(3) \quad & (u, -\zeta^2 v), (z u, -\zeta v), (z^2 u, -v) \quad && \text{for $k \geq 3$.}
\end{alignat*}
\end{lemma}

\begin{theorem}[\cite{barmeiergasparim2}]
\label{quantizable}
Let $\eta \in \H^0 (Z_k, \extprod^2 \mathcal T_{Z_k})$ be the cohomology class represented by the $0$-cocycle $(\eta_U, \eta_V) = (z u, -\zeta v)$. Then $\eta$ can be quantized, giving rise to a commutative diagram
\begin{align}
\label{diamond}
\begin{tikzpicture}[baseline=-.35em,description/.style={fill=white,inner sep=2pt}]
\matrix (m) [matrix of math nodes, row sep={3.5em,between origins}, inner sep=2pt,
column sep={3.5em,between origins}, ampersand replacement=\&]
{\& \mathcal O_{Z_k} \llrr{\hbar} (Z_k) \& \\ \mathcal O_{Z_k} \llrr{\hbar} (U) \&\& \mathcal O_{Z_k} \llrr{\hbar} (V) \\ \& \mathcal O_{Z_k} \llrr{\hbar} (U \cap V) \& \\};
\path[stealth-, line width=.65pt, font=\scriptsize]
(m-3-2) edge (m-2-1)
(m-3-2) edge (m-2-3)
(m-2-1) edge (m-1-2)
(m-2-3) edge (m-1-2)
;
\end{tikzpicture}
\end{align}
where $U \simeq \mathbb C^2$ is endowed with the Kontsevich star product associated to $\eta_U$.
\end{theorem}

Denoting by $\mu_0, \nu_0, \xi_0$ the (commutative) multiplications on
\begin{align*}
\mathcal O_{Z_k} (U) &\simeq \mathbb C [\eqmakebox[zzeta]{$z,{}$} u] \\
\mathcal O_{Z_k} (V) &\simeq \mathbb C [\eqmakebox[zzeta]{$\zeta,{}$} v] \\
\mathcal O_{Z_k} (U \mkern1mu{\cap}\mkern1mu V) &\simeq \mathbb C [\eqmakebox[zzeta]{$z^\pm,{}$} u]
\end{align*}
respectively, the deformation of the diagram $\mathcal O_{Z_k} \vert_{\mathfrak U}$ corresponding to the deformation quantization given in Theorem \ref{quantizable} can be expressed as $\mu_n = B_n^U$ with $\mu_1 = B_1^U = \{ \blank {,} \blank \}_{\eta_U}$ and similarly for $\nu_n$ and $\xi_n$.

For example, on monomials $z^a u^b, z^c u^d \in \mathbb C [z, u] \simeq \mathcal O_{Z_k} (U)$ we have
\begin{align}
\mu_1 \colon \eqmakebox[obs]{$(z^a u^b, z^c u^d)$} &\tikzmapsto \{ z^a u^b, z^c u^d \}_{\eta_U} = (ad-bc) z^{a+c} u^{b+d} \label{mu1}
\intertext{and similarly}
\nu_1 \colon \eqmakebox[obs]{$(\zeta^a v^b, \zeta^c v^d)$} &\tikzmapsto \{ \zeta^a v^b, \zeta^c v^d \}_{\eta_V} = -(ad-bc) \zeta^{a+c} v^{b+d} \label{nu1}
\end{align}
where the sign in (\ref{nu1}) appears because of the sign in $(\eta_U, \eta_V) = (zu, -\zeta v)$.

\subsection{Classical deformations}
\label{classicaldeformations}

While in \S \ref{deformationquantization} we considered deformations of the multiplication on the algebra $\mathcal O_X (U)$ of sections over an open set $U$, but fixing the morphisms, now we show how classical deformations can be considered as deformations of the morphisms, but fixing the (commutative) algebra structure over each open set.

A formal deformation of a smooth semi-separated scheme $X$ can be considered as a family over $\mathbb C \llrr{\epsilon}$. Given a cover of $X$ by smooth affine schemes $U_i$, the individual affine charts do not admit any ``classical'' scheme-theoretic deformations (see \cite{sernesi}). However, such a family gives rise to a deformation of the morphisms which we now illustrate for the noncompact surfaces $Z_k$, whose classical deformations have been studied in \cite{barmeiergasparim1}.

\begin{lemma}[{\cite[Lem.\ 5.3]{barmeiergasparim1}}]
Let $k \geq 2$. Then $\H^1 (Z_k, \mathcal T_{Z_k}) \simeq \mathbb C^{k-1}$ and in canonical coordinates a cohomology class $\theta \in \H^1 (Z_k, \mathcal T_{Z_k})$ can be represented by a $1$-cocycle
\[
\sum_{i=1}^{k-1} t_i z^{-k+i} \frac{\partial}{\partial u}
\]
for some coefficients $t_i \in \mathbb C$.
\end{lemma}

\begin{theorem}[{\cite[Thm.\ 5.4]{barmeiergasparim1}}]
Let $k \geq 2$. Then $Z_k = \Tot \mathcal O_{\mathbb P^1} (-k)$ admits a $(k{-}1)$-dimensional semi-universal family $Z_k \tikzto M \tikzto \mathbb C^{k-1} \simeq \H^1 (Z_k, \mathcal T_{Z_k})$ of ``classical deformations'', which may be constructed as the family of deformations of the vector bundle structure of $Z_k$ to an affine bundle over $\mathbb P^1$.
\end{theorem}

A classical deformation $Z_k (0, \theta)$ of $Z_k$ can now be given by the same coordinate charts, but with the identification
\begin{equation}
\label{coordinatechange}
(\zeta, v) = \Big( z^{-1}, z^k u + \textstyle\sum\limits_{i=1}^{k-1} t_i z^i \Big).
\end{equation}

Regarded as a formal deformation parametrized by $\theta$, the change of coordinates (\ref{coordinatechange}) is given by
\begin{equation}
\label{coordinatechangeepsilon}
(\zeta, v) = \Big( z^{-1}, z^k u + \epsilon \textstyle\sum\limits_{i=1}^{k-1} t_i z^i \Big).
\end{equation}

Writing
\begin{align*}
\begin{tikzpicture}[baseline=-2.6pt,description/.style={fill=white,inner sep=2pt}]
\matrix (m) [matrix of math nodes, row sep=0em, text height=1.5ex, column sep=2em, text depth=0.25ex, ampersand replacement=\&]
{\mathcal O_{Z_k} (U) \& \mathcal O_{Z_k} (U \cap V) \& \mathcal O_{Z_k} (V) \\[1em]
\mathbb C [z, u] \& \eqmakebox[Czu]{$\mathbb C [z^\pm, u]$} \& \mathbb C [\zeta, v] \\
z \& \eqmakebox[Czu]{$z$\hair, \hfill $z^{-1}$} \& \zeta \\
u \& \eqmakebox[Czu]{$u$\hair, \hfill $z^k u$}  \& v \\};
\path[->,line width=.6pt,font=\scriptsize]
(m-1-1) edge node[above=-.4ex] {$\phi_0$} (m-1-2)
(m-1-3) edge node[above=-.4ex] {$\psi_0$} (m-1-2)
(m-2-1) edge (m-2-2)
(m-2-3) edge (m-2-2)
;
\path[|->,line width=.6pt]
(m-3-3) edge (m-3-2)
(m-3-1) edge (m-3-2)
(m-4-3) edge (m-4-2)
(m-4-1) edge (m-4-2)
;
\end{tikzpicture}
\end{align*}
the coordinate change (\ref{coordinatechangeepsilon}) is defined on linear monomials $\zeta, v$ and can be generalized to an algebra homomorphism defined on arbitrary monomials $\zeta^m v^n$ by
\[
\psi = \psi_0 + \epsilon \, \psi_1 + \epsilon^2 \psi_2 + \dotsb
\]
where $\psi \colon \zeta^m v^n \tikzmapsto (z^{-1})^m \big( z^k u + \epsilon \sum_{i=1}^{k-1} t_i z^i \big)^n$. Expanding this in powers of $\epsilon$ gives the expressions for $\psi_i$; the first terms are given by the linear maps
\begin{align*}
\psi_0 \colon \zeta^m v^n &\tikzmapsto z^{nk-m} u^n = (z^{-1})^m (z^k u)^n \\
\psi_1 \colon \zeta^m v^n &\tikzmapsto \sum_{i=1}^{k-1} n \, t_i \, z^{(n-1)k-m+i} u^{n-1} \\
\psi_2 \colon \zeta^m v^n &\tikzmapsto \sum_{i,j=1}^{k-1} \tfrac{n(n-1)}2 \, t_i t_j \, z^{(n-2)k-m+i+j} u^{n-2}.
\end{align*}
Note that $\psi_1$ is precisely the $1$-cocycle $\sum_{i=1}^{k-1} t_i z^{-k+i} \frac{\partial}{\partial u}$ applied to $\psi_0 (\zeta^m v^n) = z^{nk-m} u^n$. In other words, the $1$-cocycle representing $\theta$ is precisely the first-order term of a deformation of the restriction morphism of the sheaf $\mathcal O_{Z_k}$.

A ``classical'' (commutative) deformation of the diagram $\mathcal O_{Z_k} \vert_{\mathfrak U}$ is thus given by the diagram $\mathcal O_{Z_k} \llrr{\epsilon} \vert_{\mathfrak U}$ with undeformed multiplications, but deformed morphisms
\begin{align*}
\phi &= \phi_0 \\
\psi &= \psi_0 + \epsilon \psi_1 + \epsilon^2 \psi_2 + \dotsb
\end{align*}

Collecting powers of $\epsilon$, the fact that $\psi = \psi_0 + \epsilon \psi_1 + \dotsb$ is a morphisms is equivalent to
\begin{equation}
\label{obstructionmorphism}
\psi_n (\blank \cdot \blank) = \sum_{i+j = n} \psi_i (\blank) \cdot \psi_j (\blank).
\end{equation}

\subsection{Simultaneous deformations}
\label{simultaneousdeformations}

In \S\S \ref{deformationquantization}--\ref{classicaldeformations} we saw that cohomology classes $\eta \in \H^0 (Z_k, \extprod^2 \mathcal T_{Z_k})$ and $\theta \in \H^1 (Z_k, \mathcal T_{Z_k})$ give rise to deformation quantizations and classical (commutative) deformations of $\mathcal O_{Z_k} \vert_{\mathfrak U}$.

A simple calculation gives
\[
\HH^3 (Z_k) \simeq \H^1 (Z_k, \extprod^2 \mathcal T_{Z_k}) \simeq
\begin{cases}
0 & 1 \leq k \leq 3 \\
\mathbb C^{k-3} & k \geq 4
\end{cases}
\]
so that there is an obstruction space to simultaneous deformations of $\mathcal O_{Z_k} \vert_{\mathfrak U}$ for $k \geq 4$. We now show that the continuation of a general $2$-cocycle representing $\eta \oplus \theta \in \H^0 (Z_k, \extprod^2 \mathcal T_{Z_k}) \oplus \H^1 (Z_k, \mathcal T_{Z_k}) \simeq \HH^2 (Z_k)$ to higher orders may define a non-zero obstruction class in $\HH^3 (Z_k) \simeq \H^1 (Z_k, \extprod^2 \mathcal T_{Z_k})$, even if $\eta$ and $\theta$ can individually be continued to all orders. This obstruction class can be computed from the Maurer--Cartan equation (\ref{maurercartanzk2}) for the Gerstenhaber--Schack L$_\infty$ algebra $\mathfrak{gs} (\mathcal O_{Z_k} \vert_{\mathfrak U})$ to which we will turn to next.

For simplicity we choose the $1$-cocycle $\theta = t_i z^{-k+i} \partial_u$ (rather than a sum over $1 \leq i \leq k-1$) and the Poisson structure given by the global bivector field $(z u \, \partial_z \mkern1mu {\wedge} \mkern1.5mu \partial_u, -\zeta v \, \partial_\zeta \mkern1mu {\wedge} \mkern1.5mu \partial_v)$. (Note that this Poisson structure is quantizable for the open immersions $U, V \subset Z_k$ for any $k \geq 1$, see Proposition \ref{quantizable}.)

Denoting by $\nu_0$ and $\xi_0$ the (commutative) multiplications of $\mathcal O_{Z_k} (V) \simeq \mathbb C [\zeta, v]$ and $\mathcal O_{Z_k} (U \cap V) \simeq \mathbb C [z^\pm, u]$, respectively, we check the obstruction on two arbitrary monomials $\zeta^a v^b, \zeta^c v^d \in \mathbb C [\zeta, v] = \mathcal O_{Z_k} (V)$.

\paragraph{First order.} 

The obstruction in $\epsilon$ of the form
\[
\{ \psi_0 (\blank), \psi_0 (\blank) \} + \psi_1 (\blank) \psi_0 (\blank) + \psi_0 (\blank) \psi_1 (\blank) - \psi_0 (\{ \blank {,} \blank \}) - \psi_1 (\blank \cdot \blank)
\]
is easily checked to vanish, see (\ref{obstructionmultiplication}) and (\ref{obstructionmorphism}). Here we have written ${}\cdot{}$ for $\nu_0$ on $V$ and $\{ \blank {,} \blank \}$ for both $\nu_1$ and $\xi_1$.

\paragraph{Second order.} 

Similarly, the vanishing of the obstruction in $\epsilon^2$ amounts to
\begin{align*}
\psi_2 (\blank \cdot \blank) + \psi_0 (\nu_2 (\blank {,} \blank)) + \psi_1 (\{ \blank {,} \blank \}) &= \psi_0 (\blank) \psi_2 (\blank) + \psi_2 (\blank) \psi_0 (\blank) \\
&\quad + \psi_1 (\blank) \psi_1 (\blank) + \xi_2 (\psi_0(\blank), \psi_0 (\blank)) \\
&\quad + \{ \psi_0 (\blank), \psi_1 (\blank) \} + \{ \psi_1 (\blank), \psi_0 (\blank) \}.
\intertext{When the multiplication, respectively the morphisms, are {\it not} deformed this obstruction simplifies to}
\psi_2 (\blank \cdot \blank) &= \psi_0 (\blank) \psi_2 (\blank) + \psi_2 (\blank) \psi_0 (\blank) + \psi_1 (\blank) \psi_1 (\blank) \\
\intertext{respectively}
\psi_0 (\nu_2 (\blank {,} \blank)) &= \xi_2 (\psi_0 (\blank), \psi_0 (\blank))
\end{align*}
which again hold by (\ref{obstructionmultiplication}) and (\ref{obstructionmorphism}).

It remains to check the obstruction involving higher terms of both $\psi_i$ on the one hand, and $\nu_i, \xi_i$ on the other. This obstruction reads
\begin{align}
\label{obstructionthreeterms}
\psi_1 (\{ \blank {,} \blank \}) &= \{ \psi_0 (\blank), \psi_1 (\blank) \} + \{ \psi_1 (\blank), \psi_0 (\blank) \}.
\end{align}

The individual terms applied to arbitrary monomials $\zeta^a v^b, \zeta^c v^d$ read as follows.
\begin{align*}
\psi_1 (\{ \zeta^a v^b, \zeta^c v^d \}) &= -(ad-bc) \, \psi_1 (\zeta^{a+c} v^{b+d}) \\
&= -(ad-bc) \, (b+d) \, t_i \, z^{(b+d-1)k - (a+c) + i} u^{b+d-1} \\[1.25ex]
\{ \psi_0 (\zeta^a v^b), \psi_1 (\zeta^c v^d) \} &= \xi_1 (z^{bk-a} u^b, d \, t_i \, z^{(d-1)k - c + i} u^{d-1}) \\
&= d ((bk{-}a)(d{-}1) {-} b ((d{-}1)k {-} c {+} i)) \, t_i \, z^{(b+d-1)k - (a+c) + i} u^{b+d-1} \\[1.25ex]
\{ \psi_1 (\zeta^a v^b), \psi_0 (\zeta^c v^d) \} &= \xi_1 (b \, t_i \, z^{(b-1)k - a + i} u^{b-1}, z^{dk-c} u^d) \\
&= b (((b{-}1)k {-} a {+} i) d {-} (b{-}1)(dk{-}c)) \, t_i \, z^{(b+d-1)k - (a+c) + i} u^{b+d-1}
\end{align*}
Dropping the monomials and the factor of $t_i$, the obstruction (\ref{obstructionthreeterms}) amounts to
\[
- (ad-bc)(b+d) = -(ad-bc)(b+d) + ad - bc
\]
{\it i.e.}\ $ad - bc = 0$. Of course, this is not satisfied for all $a,b,c,d \in \mathbb N$ and for $k \geq 4$ and $1 < i < k-1$ the $3$-cocycle
\[
\psi_1 (\nu_1) - \xi_1 (\psi_0 \otimes \psi_1) - \xi_1 (\psi_1 \otimes \psi_0)
\]
involving $\eta$ and $\theta$ thus defines a non-zero class in $\HH^3 (Z_k)\simeq \H^1 (Z_k, \extprod^2 \mathcal T_{Z_k})$. In other words, the simultaneous deformation in a commutative and noncommutative direction of $Z_k$ may be obstructed already at second order.

We summarize our findings in the following proposition.

\begin{proposition}
\label{nosimultaneous}
There exists no simultaneous deformation of the pair
\[
\eta \oplus \theta \in \H^0 (Z_k, \extprod^2 \mathcal T_{Z_k}) \oplus \H^1 (Z_k, \mathcal T_{Z_k}) \simeq \HH^2 (Z_k)
\]
restricting to the purely noncommutative and purely commutative deformations described in \S \ref{deformationquantization} and in \S \ref{classicaldeformations} respectively.
\end{proposition}

We end with a conceptual explanation of this observation. The (cubic) Maurer--Cartan equation (\ref{maurercartanzk2}) for $\mathfrak{gs} (\mathbb A)$ gives rise to an obstruction map $\operatorname{obs} \colon \mathrm C_{\mathrm{GS}}^2 (\mathbb A) \tikzto \mathrm C_{\mathrm{GS}}^3 (\mathbb A)$ such that the preimage of $0 \in \mathrm C_{\mathrm{GS}}^3 (\mathbb A)$ under $\operatorname{obs}$ are precisely solutions to the Maurer--Cartan equation. The solution space $M = \operatorname{obs}^{-1} (0) = \MC (\mathfrak{gs} (\mathbb A))$ can thus be viewed as a variety in $\mathrm C_{\mathrm{GS}}^2 (\mathbb A)$ cut out by polynomial equations. Under the decomposition $\mathrm C_{\mathrm{GS}}^2 (\mathbb A) = \mathrm C^{0,2} \oplus \mathrm C^{1,1}$, purely noncommutative and purely commutative deformations can be thought of as curves $C_1$ and $C_2$ in $M$ whose tangent vectors at $\mathcal O_{Z_k} \vert_{\mathfrak U}$ are $\eta$ and $\theta$, respectively. These curves lie in the intersection of $M$ with the coordinate hyperplanes $\mathrm C^{0,2}$, respectively $\mathrm C^{1,1}$, whereas Proposition \ref{nosimultaneous} shows there cannot exist a curve in $M$ with tangent $\eta \oplus \theta$ projecting to $C_1 \subset \mathrm C^{0,2}$ and $C_2 \subset \mathrm C^{1,1}$.

The question of the existence of formal simultaneous deformations of $\mathcal O_Z \vert_{\mathfrak U}$, cut out by the (cubic) Maurer--Cartan equation (\ref{maurercartanzk2}), shall be addressed in future work.

\subsection*{Acknowledgements}

The first named author learned much of the deformation theory from Wendy Lowen as well as Pieter Belmans and Hoang Dinh Van during a visit to Antwerp and it is a great pleasure to thank them all. The first named author also thanks the Studienstiftung des deutschen Volkes for support. A large part of this work was carried out at the Max Planck Institute for Mathematics in Bonn and we are grateful for the excellent resources and working conditions there.

\appendix

\section{Hochschild cohomology and simplicial presheaf cohomology}
\label{hochschildsimplicial}

References for this appendix are Dinh Van--Lowen \cite{dinhvanlowen} and Gerstenhaber--Schack \cite{gerstenhaberschack}.

\subsection{Hochschild cohomology of algebras}

\begin{definition}
\label{definitionhochschild}
Let $S$ be a $\mathbb k$-algebra and let $M$ be an $S$-bimodule. The {\it Hochschild complex} $\mathrm{CH}^\hdot (S, M)$ is defined as
\[
\mathrm{CH}^q (S, M) = \Hom_{\mathbb k} (S^{\otimes q}, M)
\]
with the {\it Hochschild differential} $d_{\mathrm H} \colon \mathrm{CH}^q (S, M) \tikzto \mathrm{CH}^{q+1} (S, M)$ given by
\begin{align*}
d_{\mathrm H} (x) (s_0, \dotsc, s_q) &= s_0 \cdot x (s_1, \dotsc, s_q) \\
& \quad {} + \sum_{i=1}^q (-1)^i x (s_0, \dotsc, s_{i-1} s_i, \dotsc, s_q) \\
& \quad {} + (-1)^{q+1} x (s_0, \dotsc, s_{q-1}) \cdot s_q.
\end{align*}
\end{definition}

\subsection{Simplicial cohomology of presheaves}

Let $\mathfrak U$ be a small category and denote by $\mathrm N = \mathrm N (\mathfrak U)$ the simplicial nerve of $\mathfrak U$. A $p$-simplex $\sigma \in \mathrm N_p$ is a string of $p$ composable morphisms in $\mathfrak U$, which we write
\[
\sigma = ( U_0 \toarg{\phi_1} U_1 \toarg{\phi_2} \dotsb \toarg{\phi_{p-1}} U_{p-1} \toarg{\phi_p} U_p ).
\]
The simplicial structure of $\mathrm N$ gives face maps
\begin{align}
\begin{tikzpicture}[baseline=-2.6pt,description/.style={fill=white,inner sep=2pt}]
\matrix (m) [matrix of math nodes, row sep=0em, text height=1.5ex, column sep=0em, text depth=0.25ex, ampersand replacement=\&, column 3/.style={anchor=base west}, column 1/.style={anchor=base east}]
{\partial_i \colon {} \&[-.6em] \mathrm N_{p+1} \&[2em] \mathrm N_p \\
\& \sigma \& \partial_i \sigma \\};
\path[->,line width=.6pt,font=\scriptsize]
(m-1-2) edge (m-1-3)
;
\path[|->,line width=.6pt]
(m-2-2) edge (m-2-3)
;
\end{tikzpicture}
\end{align}
where
\begin{flalign*}
&&
\begin{tikzpicture}[baseline=-2.6pt,description/.style={fill=white,inner sep=2pt}]
\matrix (m) [matrix of math nodes, row sep=.75em, text height=1.5ex, column sep=1em, text depth=0.25ex, ampersand replacement=\&, column 1/.style={anchor=base east}, column 2/.style={anchor=base east}, column 3/.style={anchor=base east}, column 9/.style={anchor=base west}, column 10/.style={anchor=base west}, outer sep=0, inner sep=2.2pt]
{
\partial_0     \sigma = {} \&[-1.4em]           \& \mathllap{\big(} U_1 \& \dotsb \& U_{i-1} \&
U_i \& U_{i+1} \& \dotsb \& U_p \& U_{p+1} \big) \\
\partial_i     \sigma = {} \&[-1.4em] \big( U_0 \&                  U_1 \& \dotsb \& U_{i-1} \&
    \& U_{i+1} \& \dotsb \& U_p \& U_{p+1} \big) \\
\partial_{p+1} \sigma = {} \&[-1.4em] \big( U_0 \&                  U_1 \& \dotsb \& U_{i-1} \&
U_i \& U_{i+1} \& \dotsb \& U_p \mathrlap{\big).} \\
};
\path[-stealth,line width=.6pt,font=\scriptsize]
(m-1-3) edge (m-1-4)
(m-1-4) edge (m-1-5)
(m-1-5) edge (m-1-6)
(m-1-6) edge (m-1-7)
(m-1-7) edge (m-1-8)
(m-1-8) edge (m-1-9)
(m-1-9) edge (m-1-10)
(m-2-2) edge (m-2-3)
(m-2-3) edge (m-2-4)
(m-2-4) edge (m-2-5)
(m-2-5) edge (m-2-7)
(m-2-7) edge (m-2-8)
(m-2-8) edge (m-2-9)
(m-2-9) edge (m-2-10)
(m-3-2) edge (m-3-3)
(m-3-3) edge (m-3-4)
(m-3-4) edge (m-3-5)
(m-3-5) edge (m-3-6)
(m-3-6) edge (m-3-7)
(m-3-7) edge (m-3-8)
(m-3-8) edge (m-3-9)
;
\end{tikzpicture}
&& 1 \leq i \leq p
\end{flalign*}

\begin{definition}
\label{definitionsimplicial}
Let $\mathcal F$ and $\mathcal E$ be presheaves of $\mathbb k$-modules over $\mathfrak U$ and define the {\it simplicial presheaf complex} by
\[
\mathrm C^p (\mathcal E, \mathcal F) = \prod_{U_0 \smallto \dotsb \smallto U_p} \Hom (\mathcal E (U_p), \mathcal F (U_0)).
\]
The {\it simplicial differential} is defined as
\[
d_\Delta = \sum_{i=0}^{p+1} (-1)^i d_i
\]
where $d_i$ is obtained from $\partial_i$ as follows.

An element $x \in \mathrm C^p (\mathcal E, \mathcal F)$ is a tuple $x = (x^\tau)_{\tau \in \mathrm N_p}$ of homomorphisms $x^\tau \colon \mathcal E (U_p) \tikzto \mathcal F (U_0)$ for each $\tau = (U_0 \tikzto \dotsb \tikzto U_p)$. We define $d_i x = (d_i x^\sigma)_{\sigma \in \mathrm N_{p+1}}$ where for each $\sigma = (U_0 \toarg{\phi_1} \dotsb \toarg{\phi_{p+1}} U_{p+1})$
\begin{flalign*}
&& d_0     x^\sigma &= \mathcal F (\phi_1) \circ x^{\partial_0 \sigma} && \\
&& d_i     x^\sigma &= x^{\partial_i \sigma} && \mathllap{1 \leq i \leq p} \\
&& d_{p+1} x^\sigma &= x^{\partial_{p+1} \sigma} \circ \mathcal E (\phi_{p+1}) &&
\end{flalign*}
so that each $d_i x^\sigma \colon \mathcal E (U_{p+1}) \tikzto \mathcal F (U_0)$, whence
\[
d_\Delta \colon \prod_{U_0 \smallto \dotsb \smallto U_p} \Hom (\mathcal E (U_p), \mathcal F (U_0)) \tikzto \prod_{U_0 \smallto \dotsb \smallto U_{p+1}} \Hom (\mathcal E (U_{p+1}), \mathcal F (U_0)).
\]
\end{definition}

\subsection{Differentials in the Gerstenhaber--Schack double complex}

In Definition \ref{definitiongerstenhaberschack} the Gerstenhaber--Schack complex was defined as the total complex of the first quadrant double complex
\begin{align*}
\mathrm C^{p,q} (\mathcal F) :={} &\mathrm C^p (\mathcal F^{\otimes q}, \mathcal F) \\
={} &\prod_{U_0 \smallto \dotsb \smallto U_p} \Hom (\mathcal F (U_p)^{\otimes q}, \mathcal F (U_0)).
\end{align*}
The (vertical) {\it Hochschild differential}
\[
d_{\mathrm H} \colon \prod_{U_0 \smallto \dotsb \smallto U_p} \Hom (\mathcal F (U_p)^{\otimes q}, \mathcal F (U_0)) \tikzto \prod_{U_0 \smallto \dotsb \smallto U_p} \Hom (\mathcal F (U_p)^{\otimes q+1}, \mathcal F (U_0))
\]
is now given on each component, {\it i.e.}\ for each $p$-simplex
\[
U_0 \toarg{\phi_1} U_1 \toarg{\phi_2} \dotsb \toarg{\phi_{p-1}} U_{p-1} \toarg{\phi_p} U_p,
\]
as in Definition \ref{definitionhochschild} by regarding $\mathcal F (U_0)$ as an $\mathcal F (U_p)$-bimodule with left and right actions given by left- respectively right-multiplication in $\mathcal F (U_0)$
\begin{flalign*}
&& \begin{aligned}
s' \cdot s\phantom{'}  &= \mathcal F (\phi_p \circ \dotsb \circ \phi_1) (s') s \\
s  \cdot s' &= s (\mathcal F (\phi_p) \circ \dotsb \circ \mathcal F (\phi_1)) (s')
\end{aligned}
&& \quad\mathllap{s' \in \mathcal F (U_p), s \in \mathcal F (U_0).}
\end{flalign*}
The (horizontal) {\it simplicial differential}
\[
d_\Delta \colon \mathrm C^{p,q} (\mathcal F) \tikzto \mathrm C^{p+1,q} (\mathcal F)
\]
is given by taking $\mathcal E = \mathcal F^{\otimes q}$ in Definition \ref{definitionsimplicial}.

\end{document}